\documentclass[11pt]{article}

\usepackage[T1]{fontenc}
\usepackage{lmodern}
\usepackage{amsmath,amsfonts,amsthm}
\usepackage{graphicx}
\usepackage{booktabs}
\usepackage[margin=1in]{geometry}
\usepackage[hidelinks]{hyperref}

\allowdisplaybreaks

\newtheorem{theorem}{Theorem}[section]
\newtheorem{lemma}[theorem]{Lemma}
\newtheorem{proposition}[theorem]{Proposition}
\newtheorem{corollary}[theorem]{Corollary}
\theoremstyle{definition}
\newtheorem{example}[theorem]{Example}
\theoremstyle{remark}
\newtheorem{remark}[theorem]{Remark}

\newcommand{\R}{\mathbb R}
\newcommand{\E}{\mathbb E}
\newcommand{\Var}{\operatorname{Var}}
\newcommand{\SCV}{\operatorname{SCV}}
\newcommand{\Pcal}{\mathcal P}
\newcommand{\Vcal}{\mathcal V}
\newcommand{\dd}{\,\mathrm d}
\newcommand{\di}{\mathrm d}
\newcommand{\e}{\mathrm e}

\hypersetup{
 pdftitle={Least Variable Matrix-Exponential Distributions with Real Spectrum},
 pdfauthor={Maria Laura Battagliola and Oscar Peralta}
}

\title{Least Variable Matrix-Exponential Distributions\\ with Real Spectrum}
\author{Maria Laura Battagliola\thanks{Department of Statistics, Instituto Tecnol\'ogico Aut\'onomo de M\'exico (ITAM), Mexico City, Mexico. \texttt{laura.battagliola@itam.mx}} \and Oscar Peralta\thanks{Department of Mathematical Sciences, University of Copenhagen, Copenhagen, Denmark. \texttt{opg@math.ku.dk}}}
\date{}

\begin{document}
\maketitle

\begin{abstract}
Aldous and Shepp proved that the Erlang law minimizes variance, at fixed mean, among phase-type distributions with at most $N$ phases. We solve the corresponding problem for matrix-exponential distributions of order at most $N$ whose Laplace transforms have only real poles, a class containing the $N$-phase Erlang law. We establish a positive-mixture representation of these distributions in terms of single-pole densities of polynomial-square type, without increasing their order. This decomposition reduces variance minimization to a polynomial problem, which we solve explicitly using Gauss--Laguerre quadrature. We determine the minimum variance and characterize all minimizing distributions. Every minimizer has a single-pole polynomial-square density of order $N$, whose polynomial factor is obtained from a generalized Laguerre polynomial by removing two adjacent zeros of minimal ratio. For $N\geq3$, none of these minimizers admits a phase-type representation. At unit mean, the minimum variance is $\pi^2/N^2-8\pi^2/N^3+O(N^{-4})$, compared with $1/N$ for Erlang distributions. At every order, at most two adjacent gaps need to be compared to identify a minimizing pair, allowing optimizers to be computed at large orders.
\end{abstract}



\section{Introduction}\label{sec:introduction}

Replacing a fixed time horizon by an independent random time can simplify the computation of probabilities and expectations. In the \emph{Erlangization method}, the time $t_0>0$ is replaced by $t_0U_N$, where $U_N$ has the mean-one Erlang distribution with $N$ phases. Its variance is $1/N$, so the random horizon converges to $t_0$ in mean square, while the finite-state Markovian representation permits evaluation by matrix methods. This approach has been used for finite-time ruin probabilities \cite{asmussenavramusabel2002,bladtnielsenperalta2019}, with related randomization methods appearing in option valuation \cite{carr1998,bouchardelkarouitouzi2005}. Further uses include Canadization through fluctuation theory \cite{kyprianoupistorius2003,avramkyprianoupistorius2004} and Wiener--Hopf simulation for L\'evy processes \cite{kuznetsovetal2011}. Fluctuation identities for one-sided L\'evy processes also extend from exponential to more general random horizons, providing a direct setting for the use of low-variance clocks \cite{bladtivanovs2021}.

How small can the variance be if the random horizon has mean one and at most $N$ phases? Aldous and Shepp \cite{aldousshepp1987} answered this question for phase-type (PH) distributions. Among mean-one absorption times represented by a Markov jump process with at most $N$ transient states, the mean-one Erlang law with $N$ phases minimizes the variance, with value $1/N$. O'Cinneide \cite{ocinneide1991} strengthened this conclusion to a majorization result and gave an example of a distribution with a rational Laplace transform and a smaller coefficient of variation than the Erlang law of the same order. Determining the minimal variability in that wider setting has remained open, as the Markovian arguments of \cite{aldousshepp1987,ocinneide1991} do not carry over.

This larger family consists of the matrix-exponential (ME) distributions, which extend the PH class and retain many of its matrix methods \cite[Chs.~3--4]{bladtnielsen2017}. Among distributions on $[0,\infty)$ without an atom at zero, the ME laws are precisely those whose Laplace transform is a strictly proper rational function. The \emph{order} of an ME law is the degree of the reduced denominator of its Laplace transform, or equivalently the minimal dimension of an ME representation. This order can be smaller than the number of states required by a PH representation. Moreover, ME laws need not admit an absorbing finite-state Markov representation, although alternative probabilistic constructions are available \cite{peralta2023}.

We call an ME law \emph{real-spectrum} if all poles of its Laplace transform are real. This class contains Coxian, hyperexponential, and Erlang laws, as well as their finite mixtures and convolutions, though it excludes PH laws with nonreal poles. Does the $N$-phase Erlang law still minimize variance within this class? Since rescaling a real-spectrum ME variable preserves its order and keeps its poles real, minimizing variance at unit mean is equivalent to minimizing the squared coefficient of variation (SCV), given by $\Var(X)/\E[X]^2$, over all real-spectrum ME random variables $X$ of order at most $N$. We therefore study
\begin{equation}\label{eq:optimization}
\mathfrak v_N^{\mathrm R}=\inf\bigl\{\Var(U)\mid\E[U]=1,\ U\text{ is real-spectrum ME of order at most }N\bigr\}.
\end{equation}
Although ME distributions are sometimes defined to allow an atom at zero, this possibility is inconsequential for \eqref{eq:optimization}. Indeed, if $U$ equals zero with probability $p\in(0,1)$ and otherwise follows an ME law $G$ without an atom at zero, then $G$ has the same order and poles as $U$, while $\Var(U)/\E[U]^2=(1+v_G)/(1-p)-1>v_G$, where $v_G$ is the SCV of $G$. Allowing an atom at zero therefore neither lowers the infimum nor produces a minimizer, so throughout we restrict to ME laws without an atom at zero.

The Erlang law remains optimal at orders one and two, whereas at every higher order the minimum is strictly smaller than $1/N$. To find it, write the positive integer $N$ uniquely as $2m+1+\alpha$, with $m\in\mathbb Z_{\geq0}$ and $\alpha\in\{0,1\}$. All polynomial-degree indices below are nonnegative integers. We use the generalized Laguerre polynomials in the standard normalization, given for $n\geq0$ by
\begin{equation}\label{eq:laguerre}
L_n^{(\alpha)}(t)=\sum_{k=0}^n(-1)^k\binom{n+\alpha}{n-k}\frac{t^k}{k!}.
\end{equation}
Let $0<x_{1,m+2}^{(\alpha)}<\cdots<x_{m+2,m+2}^{(\alpha)}$ denote the simple, strictly positive zeros of $L_{m+2}^{(\alpha)}$, abbreviated $x_j=x_{j,m+2}^{(\alpha)}$. Define the smallest relative gap between adjacent zeros by
\begin{equation}\label{eq:relative-gap}
\Delta_{m+2}^{(\alpha)}=\min_{1\leq j<m+2}\frac{(x_{j+1}-x_j)^2}{4x_jx_{j+1}}.
\end{equation}
We prove that $\mathfrak v_N^{\mathrm R}=\Delta_{m+2}^{(\alpha)}$. To construct the optimizing density associated with a minimizing pair $x<y$, divide $L_{m+2}^{(\alpha)}(t)$ by $(t-x)(t-y)$, square the quotient, and multiply by $t^\alpha\e^{-t}$. Normalizing the mass and rescaling to mean one gives an optimizer, and every optimizer arises in this way. As $N$ grows,
\[
\mathfrak v_N^{\mathrm R}=\frac{\pi^2}{N^2}-\frac{8\pi^2}{N^3}+O(N^{-4}).
\]
We also localize the minimizing pair at every finite order to at most two adjacent gaps between three consecutive Laguerre zeros, making the computation of optimizers feasible at orders as large as $10^{100}+1$.

A key structural ingredient behind this characterization is a positive-mixture representation of independent interest. Every real-spectrum ME density of order at most $N$ can be decomposed as a finite positive mixture of components having the maximal number of nonnegative zeros (counted with multiplicity) permitted by their orders, and each such component can in turn be written as a positive rate mixture of single-pole densities proportional to $t^{\alpha_j}\e^{-\lambda t}\psi_j(t)^2$, with $\alpha_j\in\{0,1\}$ and $\psi_j$ a real polynomial, without increasing the order. This rate-mixture representation generalizes that of sums of independent exponentials as mixtures of Erlang densities of the same order. For each component in the finite mixture, the polynomial factor is determined by its nonnegative zeros, counted with multiplicity, and dividing the density by this factor yields a completely monotone function. Because the SCV of a positive mixture is at least the smallest SCV among its components, this representation reduces the global multi-pole problem \eqref{eq:optimization} directly to single-pole polynomial-square densities.

Low-variance ME distributions at fixed mean are often called \emph{concentrated}. For real-spectrum ME laws, M\'esz\'aros and Telek \cite{meszarostelek2022} investigated concentration numerically. We solve the variance minimization problem over all real-spectrum ME laws of order at most $N$, show that every optimizer has a single-pole polynomial-square density, and identify all minimizing polynomial factors. For nonreal poles, an explicit optimization-free construction is given in \cite{battagliolaperalta2026}, achieving $O((\log N)^3/N^2)$ SCV decay along its sequence of ME orders. That construction does not address exact variance minimization, and its SCV is below the Erlang benchmark at all sufficiently large orders in that sequence and above it at some smaller orders reported there. By contrast, we address both exact finite-order optimality and sharp asymptotic behavior.

The paper is organized as follows. Section~\ref{sec:results} introduces the single-pole polynomial-square family and states the main results, including the positive-mixture representation, the characterization of least variable laws, and the two-sided bounds and localization. Section~\ref{sec:single-pole} solves the variance minimization problem for single-pole polynomial-square densities using Gauss--Laguerre quadrature. Section~\ref{sec:reduction} establishes the SCV comparison for mixtures, derives the positive-mixture representation, and proves Theorem~\ref{thm:optimizers}. Section~\ref{sec:asymptotics} uses the Laguerre differential equation to establish two-sided bounds, derive the asymptotic expansion, and locate the minimizing pairs of zeros. Finally, Section~\ref{sec:computation} gives a constructive procedure for computing the optimizers at a specified order, establishes the finite-order localization result that substantially reduces the computational effort, and presents numerical illustrations.

\section{Main results}\label{sec:results}

An ME distribution without an atom at zero admits a density of the matrix form $f(t)=\boldsymbol{u}^{\mathsf T}\e^{\boldsymbol{A}t}\boldsymbol{v}$, $t\geq0$, for some integer $d\geq1$, vectors $\boldsymbol{u},\boldsymbol{v}\in\R^d$, and matrix $\boldsymbol{A}\in\R^{d\times d}$ \cite[Sec.~4.1]{bladtnielsen2017}. Here we focus on the real-spectrum case, where all poles of the Laplace transform are real, and hence strictly negative. Rather than working with the matrix form, we use the exponential-polynomial representation obtained by partial fraction expansion of the Laplace transform, writing the density as
\[
f(t)=\sum_{i=1}^k P_i(t)\e^{-p_it},\qquad t\geq0,
\]
where the rates $p_i>0$ are distinct and the $P_i$ are nonzero real polynomials. Its Laplace transform has poles at $-p_i$ of multiplicity $\deg P_i+1$, so the ME order is $\sum_{i=1}^k(\deg P_i+1)$. The coefficients of the polynomials $P_i$ may have either sign, provided that $f$ is nonnegative on $[0,\infty)$.

For any nonnegative measurable function $g$ on $[0,\infty)$ with positive finite mass and finite second moment, we define its SCV by
\[
\SCV(g)=\frac{\bigl(\int_0^\infty g(t)\dd t\bigr)\bigl(\int_0^\infty t^2g(t)\dd t\bigr)}{\bigl(\int_0^\infty tg(t)\dd t\bigr)^2}-1,
\]
which coincides with the SCV of the normalized probability density $g/\!\int_0^\infty g(t)\dd t$ and is invariant under positive scalar multiplication and changes of time scale.

Fix $N=2m+1+\alpha$, where $m\geq0$ and $\alpha\in\{0,1\}$, and let $\Pcal_m$ denote the space of real polynomials of degree at most $m$. The \emph{single-pole polynomial-square} densities that we need are obtained by normalizing and rescaling functions of the form $t^\alpha\e^{-t}b(t)^2$, where $b\in\Pcal_m$ is nonzero. Write their unnormalized moments as
\begin{equation}\label{eq:raw-moments}
M_k^{(\alpha)}(b)=\int_0^\infty t^{k+\alpha}\e^{-t}b(t)^2\dd t,\qquad k\geq0.
\end{equation}
Dividing by the mass $M_0^{(\alpha)}(b)$ and changing scale by an exponential rate $\lambda>0$ gives the probability density
\begin{equation}\label{eq:single-pole}
\kappa_{\lambda,b}^{(\alpha)}(u)=\frac{\lambda(\lambda u)^\alpha\e^{-\lambda u}b(\lambda u)^2}{M_0^{(\alpha)}(b)},\qquad u\geq0.
\end{equation}
If $b$ has degree $d\leq m$, then $t^\alpha b(t)^2$ has degree $2d+\alpha$, so $\kappa_{\lambda,b}^{(\alpha)}$ has exact ME order $2d+1+\alpha\leq N$. Indeed, the Laplace transform of $u^r\e^{-\lambda u}$ is $r!/(s+\lambda)^{r+1}$, and the nonzero leading coefficient of $t^\alpha b(t)^2$ prevents cancellation of the highest-order pole. To choose the scale that gives mean one, let $T_b^{(\alpha)}$ have density $\kappa_{1,b}^{(\alpha)}$. Then $T_b^{(\alpha)}/\lambda$ has density $\kappa_{\lambda,b}^{(\alpha)}$, with moments
\begin{equation}\label{eq:scaled-moments}
\E\bigl[(T_b^{(\alpha)}/\lambda)^k\bigr]=\frac{M_k^{(\alpha)}(b)}{\lambda^kM_0^{(\alpha)}(b)},\qquad k\geq0.
\end{equation}
For each fixed $b$, the mean-one rate is $M_1^{(\alpha)}(b)/M_0^{(\alpha)}(b)$. Choosing the best shape within this family therefore amounts to minimizing the SCV of $T_b^{(\alpha)}$, denoted by
\begin{equation}\label{eq:scv}
V_\alpha(b)=\frac{M_0^{(\alpha)}(b)M_2^{(\alpha)}(b)}{M_1^{(\alpha)}(b)^2}-1=\SCV\bigl(\kappa_{1,b}^{(\alpha)}\bigr).
\end{equation}
We can recover every real-spectrum ME density of order at most $N$ by mixing densities of the form \eqref{eq:single-pole}. The exponent $\alpha$ and the degree of $b$ may vary with the component, and each component has order at most $N$. The precise statement, proved in Section~\ref{sec:reduction}, is as follows.

\begin{theorem}[Positive-mixture representation]\label{thm:decomposition}
Let $f$ be a real-spectrum ME density of order at most $N\geq1$. For some integer $J\geq1$, we can write
\begin{equation}\label{eq:decomposition}
f(t)=\sum_{j=1}^J\omega_j f_j(t),\qquad f_j(t)=\int_0^\infty\kappa_{\lambda,b_{j,\lambda}}^{(\alpha_j)}(t)\,\Pi_j(\di\lambda),
\end{equation}
where $\omega_j>0$ and $\sum_{j=1}^J\omega_j=1$. Each $f_j$ is a real-spectrum ME density of exact order $r_j=2m_j+1+\alpha_j\leq N$, with $m_j\geq0$ and $\alpha_j\in\{0,1\}$, and has exactly $r_j-1$ zeros in $[0,\infty)$, counted with multiplicity. Here $b_{j,\lambda}(z)=\psi_j(z/\lambda)$, where $\psi_j$ is a real polynomial of degree $m_j$ with all zeros in $[0,\infty)$, and $\Pi_j$ is a probability measure supported on the closed interval between the smallest and largest exponential rates of $f_j$. Moreover, $\Pi_j$ is a point mass if and only if $f_j$ has a single pole, and the reduced Laplace denominator of each $f_j$ divides that of $f$.
\end{theorem}

\begin{remark}\label{rem:rate-mixture}
For each fixed $\lambda$, the conditional density $\kappa_{\lambda,b_{j,\lambda}}^{(\alpha_j)}$ has the single pole $-\lambda$, whereas the rate mixture $f_j$ may have several poles. Because $b_{j,\lambda}(\lambda t)=\psi_j(t)$, the polynomial factor stays fixed in the time variable $t$, so every conditional density in the mixture shares the same zeros. Although a general rate mixture need not have a rational Laplace transform, in Section~\ref{sec:reduction} we start from the poles and zeros of $f_j$ and construct a rate law $\Pi_j$ that reproduces $f_j$ exactly.
\end{remark}

Recall that $\Delta_{m+2}^{(\alpha)}$ is the smallest relative gap between adjacent zeros of the generalized Laguerre polynomial $L_{m+2}^{(\alpha)}$, as defined in \eqref{eq:relative-gap}. The following theorem gives the exact minimum and classifies all optimizers; its proof is completed in Section~\ref{sec:reduction} using the single-pole optimization of Section~\ref{sec:single-pole}.

\begin{theorem}[Least variable real-spectrum ME laws]\label{thm:optimizers}
Let $N=2m+1+\alpha$, with $m\geq0$ and $\alpha\in\{0,1\}$. The infimum in \eqref{eq:optimization} is attained and equals
\begin{equation}\label{eq:minimum}
\mathfrak v_N^{\mathrm R}=\Delta_{m+2}^{(\alpha)}.
\end{equation}
Choose any index $j^*\in\{1,\ldots,m+1\}$ attaining the minimum in \eqref{eq:relative-gap}, and put $x=x_{j^*,m+2}^{(\alpha)}<y=x_{j^*+1,m+2}^{(\alpha)}$. Define $b^*\in\Pcal_m$ and $\lambda^*>0$ by
\begin{equation}\label{eq:optimizer-shape}
b^*(t)=\frac{L_{m+2}^{(\alpha)}(t)}{(t-x)(t-y)},\qquad \lambda^*=\frac{2xy}{x+y}.
\end{equation}
The corresponding mean-one optimizer has density
\begin{equation}\label{eq:optimizer-density}
\kappa_N^*(u)=\frac{\lambda^*}{M_0^{(\alpha)}(b^*)}(\lambda^*u)^\alpha\e^{-\lambda^*u}b^*(\lambda^*u)^2,\qquad u\geq0.
\end{equation}
As $j^*$ ranges over the minimizing indices, these densities give all mean-one optimizers. For all $N\geq2$,
\begin{equation}\label{eq:order-monotonicity}
\mathfrak v_N^{\mathrm R}\leq\frac{\mathfrak v_{N-1}^{\mathrm R}}{1+\mathfrak v_{N-1}^{\mathrm R}}<\mathfrak v_{N-1}^{\mathrm R}.
\end{equation}
\end{theorem}

\begin{corollary}\label{cor:non-ph}
For every $N\geq3$, no minimizing law in Theorem~\ref{thm:optimizers} admits a PH representation of any dimension, and $\mathfrak v_N^{\mathrm R}<1/N$.
\end{corollary}

\begin{proof}
For $N\geq3$, each optimizer $\kappa_N^*$ has $m\geq1$ strictly positive interior zeros, inherited from the roots of $L_{m+2}^{(\alpha)}$ retained in $b^*$, and no atom at zero. By contrast, any PH law without an atom at zero has a continuous density strictly positive on $(0,\infty)$ \cite{ocinneide1990}, as any finite path to absorption of positive probability contributes a strictly positive convolution of exponential densities throughout $(0,\infty)$. The interior zeros of $\kappa_N^*$ therefore rule out a PH representation. By Theorem~\ref{thm:optimizers} every mean-one optimizer is of this form, so the mean-one $N$-phase Erlang law, which is PH and admissible, is not optimal. This yields $\mathfrak v_N^{\mathrm R}<1/N$.
\end{proof}

\begin{remark}\label{rem:uniqueness}
Every optimizing density in Theorem~\ref{thm:optimizers} has exact ME order $N$ and a single pole of multiplicity $N$, because $b^*$ has degree $m$. Distinct minimizing pairs have distinct mean-one rates $2xy/(x+y)$, as shown at the end of the proof of Theorem~\ref{thm:optimizers}, and therefore yield distinct optimizers; hence the minimizing law is unique exactly when the minimizing adjacent pair is unique. Numerical computations in Section~\ref{sec:computation} exhibit a unique minimizing pair at every tested order, though proving uniqueness for general $N$ remains open. For $N\in\{1,2\}$ we have $m=0$, so $b^*$ is constant and $\kappa_N^*$ is the mean-one Erlang density with $N$ phases. Indeed, the zeros of $L_2^{(0)}$ are $2\pm\sqrt2$ and those of $L_2^{(1)}$ are $3\pm\sqrt3$, so $\Delta_2^{(0)}=1$ and $\Delta_2^{(1)}=1/2$.
\end{remark}

Which adjacent zeros give the smallest relative gap when $N$ is large? Both lie near $N+4$, and their distance tends to $2\pi$. Inserted into $(y-x)^2/(4xy)$, these two facts explain the leading constant $\pi^2$ in the variance asymptotic. The precise bounds, asymptotic expansion, and localization estimates, proved in Section~\ref{sec:asymptotics}, are stated next.

\begin{theorem}[Two-sided bounds and localization]\label{thm:localization}
Write $N=2m+1+\alpha$, with $m\geq0$ and $\alpha\in\{0,1\}$. Then
\begin{equation}\label{eq:variance-bounds}
\sinh^2\!\left(\frac{\pi}{\sqrt{(N+4)^2-\alpha^2}}\right)\leq\mathfrak v_N^{\mathrm R}<\sinh^2\!\left(\frac{\pi}{\sqrt{(N+2)^2-\alpha^2}}\right),
\end{equation}
where the lower bound holds for all $N\geq1$ and the upper bound holds for all $N\geq9$. More precisely, as $N\to\infty$,
\begin{equation}\label{eq:variance-expansion}
\mathfrak v_N^{\mathrm R}=\frac{\pi^2}{N^2}-\frac{8\pi^2}{N^3}+O(N^{-4}).
\end{equation}
Fix $\alpha\in\{0,1\}$, and choose at each $m$ any minimizing index $j^*=j^*(m)$ for $\Delta_{m+2}^{(\alpha)}$, putting $x_m=x_{j^*,m+2}^{(\alpha)}<y_m=x_{j^*+1,m+2}^{(\alpha)}$. Then, as $m\to\infty$,
\begin{equation}\label{eq:zero-localization}
x_m=N+4+O(1),\qquad y_m=N+4+O(1),\qquad y_m-x_m\longrightarrow2\pi.
\end{equation}
\end{theorem}

\section{Single-pole optimization}\label{sec:single-pole}

Fix $m\geq0$ and $\alpha\in\{0,1\}$. We begin with the single-pole polynomial-square problem of minimizing $V_\alpha(b)$ in \eqref{eq:scv} over the nonzero polynomials $b\in\Pcal_m$. We evaluate the three moments in \eqref{eq:scv} by Gauss--Laguerre quadrature.

\subsection{Quadrature and discrete moments}
For any integer $n\geq1$, $n$-node Gauss--Laguerre quadrature replaces the continuous probability measure $t^\alpha\e^{-t}\dd t$ on $(0,\infty)$ (an exponential law for $\alpha=0$ and a two-phase Erlang law for $\alpha=1$) by a discrete probability distribution supported on $n$ points, called \emph{nodes}, that matches its moments through order $2n-1$, and hence integrates every polynomial of degree at most $2n-1$ exactly. These nodes are the simple positive zeros of the generalized Laguerre polynomial $L_n^{(\alpha)}$. In \eqref{eq:raw-moments}, each $M_k^{(\alpha)}(b)$ integrates the polynomial $t^kb(t)^2$ against $t^\alpha\e^{-t}\dd t$. Whenever $\deg b\leq m$ and $k\in\{0,1,2\}$, $t^kb(t)^2$ has degree at most $2m+2$, so $2n-1\geq2m+2$ guarantees exact evaluation of all three integrals for $b\in\Pcal_m$. The smallest $n$ satisfying this degree bound is $m+2$. Thus $M_0^{(\alpha)}(b)$, $M_1^{(\alpha)}(b)$, and $M_2^{(\alpha)}(b)$ can all be evaluated on the same $m+2$ zeros of $L_{m+2}^{(\alpha)}$, which depend only on $m$ and $\alpha$ and not on $k$ or on the candidate polynomial $b$. Recall that $L_{m+2}^{(\alpha)}$, defined in \eqref{eq:laguerre}, is alternatively characterized up to a nonzero scalar multiple as the polynomial of degree $m+2$ satisfying the orthogonality relation
\[
\int_0^\infty L_{m+2}^{(\alpha)}(t)p(t)t^\alpha\e^{-t}\dd t=0,\qquad p\in\Pcal_{m+1}.
\]
This orthogonality makes the $(m+2)$-node quadrature rule exact up to degree $2m+3$. Abbreviate $x_r=x_{r,m+2}^{(\alpha)}$ for the zeros $0<x_1<\cdots<x_{m+2}$ of $L_{m+2}^{(\alpha)}$, and put $C_{m+2,\alpha}=(m+2+\alpha)!/(m+2)!$, so that $C_{m+2,0}=1$ and $C_{m+2,1}=m+3$. The Gauss--Laguerre quadrature rule is
\begin{equation}\label{eq:quadrature}
\int_0^\infty h(t)t^\alpha\e^{-t}\dd t=\sum_{r=1}^{m+2}w_rh(x_r),\qquad w_r=\frac{C_{m+2,\alpha}}{x_r\bigl((L_{m+2}^{(\alpha)})'(x_r)\bigr)^2},
\end{equation}
for every polynomial $h$ of degree at most $2(m+2)-1=2m+3$ \cite[Theorems~1.45--1.46]{gautschi2004}. The expression for $w_r$ is the Laguerre weight formula \cite[Eq.~(15.3.5)]{szego1975}. Here the quadrature weights $w_r$ depend only on $m$ and $\alpha$, not on $h$, though we suppress this dependence in the notation.

For $0\neq b\in\Pcal_m$, applying the quadrature rule \eqref{eq:quadrature} to $b^2$ gives $\sum_{r=1}^{m+2}w_rb(x_r)^2=M_0^{(\alpha)}(b)$. Since the quadrature weights are positive, normalizing these terms defines the law of a random variable $Y_b$ taking values in $\{x_1,\ldots,x_{m+2}\}$ with
\begin{equation}\label{eq:node-law}
\mathbb P(Y_b=x_r)=\frac{w_rb(x_r)^2}{M_0^{(\alpha)}(b)},\qquad 1\leq r\leq m+2.
\end{equation}
Recall that $T_b^{(\alpha)}$ has density $\kappa_{1,b}^{(\alpha)}$, with moments given by \eqref{eq:scaled-moments} at $\lambda=1$. Taking $h(t)=t^kb(t)^2$ in \eqref{eq:quadrature} then yields
\[
\E\bigl[(T_b^{(\alpha)})^k\bigr]=\frac{M_k^{(\alpha)}(b)}{M_0^{(\alpha)}(b)}=\sum_{r=1}^{m+2}x_r^k\frac{w_rb(x_r)^2}{M_0^{(\alpha)}(b)}=\E[Y_b^k],\qquad k=0,1,2.
\]
In particular, the continuous variable $T_b^{(\alpha)}$ and the finite-valued variable $Y_b$ have the same SCV, so we can study the objective through the node probabilities \eqref{eq:node-law}.

As $b$ varies over the nonzero polynomials in $\Pcal_m$, the law of $Y_b$ cannot be an arbitrary probability distribution on $\{x_1,\ldots,x_{m+2}\}$. For instance, because a nonzero polynomial of degree at most $m$ has at most $m$ roots, it cannot vanish at $m+1$ of the $m+2$ distinct nodes. By \eqref{eq:node-law}, $Y_b$ therefore assigns positive probability to at least two nodes and can never concentrate at a single point. More generally, a polynomial $b\in\Pcal_m$ has only $m+1$ coefficients, yet \eqref{eq:node-law} depends on its values at $m+2$ nodes. Because an arbitrary list of $m+2$ values determines a polynomial of degree at most $m+1$, requiring $\deg b\leq m$ imposes one linear constraint among $b(x_1),\ldots,b(x_{m+2})$, namely that the coefficient of $t^{m+1}$ in the interpolating polynomial vanish. To make this constraint explicit, use Lagrange interpolation at $x_1,\ldots,x_{m+2}$ to write
\begin{equation}\label{eq:interpolation}
b(t)=\sum_{r=1}^{m+2}s_r\frac{L_{m+2}^{(\alpha)}(t)}{t-x_r},
\end{equation}
where
\begin{equation}\label{eq:coordinates}
s_r=\frac{b(x_r)}{(L_{m+2}^{(\alpha)})'(x_r)},\qquad 1\leq r\leq m+2.
\end{equation}
Indeed, because $t-x_r$ is monic, each quotient $L_{m+2}^{(\alpha)}(t)/(t-x_r)$ is a polynomial of degree $m+1$ with the same nonzero leading coefficient as $L_{m+2}^{(\alpha)}$. It vanishes at $x_j$ for $j\neq r$ and equals $(L_{m+2}^{(\alpha)})'(x_r)$ at $x_r$. Hence the right-hand side of \eqref{eq:interpolation} has degree at most $m+1$ and agrees with $b$ at the $m+2$ distinct nodes, so the two polynomials coincide. Consequently, the coefficient of $t^{m+1}$ on the right-hand side of \eqref{eq:interpolation} equals that leading coefficient multiplied by $\sum_{r=1}^{m+2}s_r$. Since $\deg b\leq m$, this coefficient must vanish, giving
\begin{equation}\label{eq:cancellation}
\sum_{r=1}^{m+2}s_r=0.
\end{equation}
Conversely, any nonzero real vector $\boldsymbol{s}=(s_1,\ldots,s_{m+2})$ satisfying \eqref{eq:cancellation} defines, through \eqref{eq:interpolation}, a polynomial $b$ of degree at most $m$ with $b(x_r)=s_r(L_{m+2}^{(\alpha)})'(x_r)$ for every $r$. Since the zeros are simple, $b\neq0$ and \eqref{eq:coordinates} returns the same $\boldsymbol{s}$. Thus $b\mapsto \boldsymbol{s}$ is a bijection from the admissible polynomials $0\neq b\in\Pcal_m$ onto the nonzero real zero-sum vectors.

Combining $b(x_r)=s_r(L_{m+2}^{(\alpha)})'(x_r)$ with the formula for $w_r$ in \eqref{eq:quadrature} cancels the squared derivatives of $L_{m+2}^{(\alpha)}$, yielding $w_rb(x_r)^2=C_{m+2,\alpha}s_r^2/x_r$ and hence
\begin{equation}\label{eq:discrete-moments}
\begin{aligned}
M_0^{(\alpha)}(b)&=C_{m+2,\alpha}\sum_{r=1}^{m+2}\frac{s_r^2}{x_r},\\
M_1^{(\alpha)}(b)&=C_{m+2,\alpha}\sum_{r=1}^{m+2}s_r^2,\qquad M_2^{(\alpha)}(b)=C_{m+2,\alpha}\sum_{r=1}^{m+2}x_rs_r^2.
\end{aligned}
\end{equation}
Let $S_s=\sum_{r=1}^{m+2}s_r^2>0$, define a random index $I$ on $\{1,\ldots,m+2\}$ by $\mathbb P(I=r)=s_r^2/S_s$, and set $\widetilde Y_b=x_I$. With this notation, \eqref{eq:discrete-moments} becomes
\[
M_0^{(\alpha)}(b)=C_{m+2,\alpha}S_s\E[\widetilde Y_b^{-1}],\qquad M_1^{(\alpha)}(b)=C_{m+2,\alpha}S_s,\qquad M_2^{(\alpha)}(b)=C_{m+2,\alpha}S_s\E[\widetilde Y_b].
\]
Substituting these identities into \eqref{eq:scv} cancels the common factor $C_{m+2,\alpha}S_s$, leaving
\begin{equation}\label{eq:size-bias}
1+V_\alpha(b)=\E[\widetilde Y_b]\E[\widetilde Y_b^{-1}].
\end{equation}
Probabilistically, $\widetilde Y_b$ has the size-biased law of $Y_b$. Indeed, since $C_{m+2,\alpha}s_r^2=x_rw_rb(x_r)^2$ and $C_{m+2,\alpha}S_s=M_1^{(\alpha)}(b)=M_0^{(\alpha)}(b)\E[Y_b]$, comparing with \eqref{eq:node-law} gives
\[
\mathbb P(\widetilde Y_b=x_r)=\frac{s_r^2}{S_s}=\frac{x_r\mathbb P(Y_b=x_r)}{\E[Y_b]}.
\]

Because \eqref{eq:cancellation} rules out a single nonzero entry, the simplest admissible vector $\boldsymbol{s}$ has only two nonzero entries, say at $i<j$. By \eqref{eq:cancellation} these entries satisfy $s_i=-s_j$, so $s_i^2=s_j^2$ and $\widetilde Y_b$ assigns mass $1/2$ to each of $x_i$ and $x_j$. Substituting into \eqref{eq:size-bias} gives
\[
V_\alpha(b)=\frac14(x_i+x_j)(x_i^{-1}+x_j^{-1})-1=\frac{(x_i+x_j)^2}{4x_ix_j}-1=\frac{(x_j-x_i)^2}{4x_ix_j}.
\]
To determine whether a vector with more nonzero entries can do better than these two-entry vectors, we need a bound that uses the cancellation condition \eqref{eq:cancellation} without assuming that only two entries of $\boldsymbol{s}$ are nonzero.

\subsection{Cancellation and index variance}

To compare all admissible vectors with the two-node candidates, we first express the objective using two independent copies of the size-biased node variable $\widetilde Y_b$. For $x,y>0$ define
\[
G(x,y)=\frac{(y-x)^2}{4xy}=\frac14\left(\sqrt{\frac yx}-\sqrt{\frac xy}\right)^2=\sinh^2\!\left(\frac12\log\frac yx\right).
\]
Let $I'$ be an independent copy of $I$ and set $\widetilde Y_b'=x_{I'}$. By independence, $\E[\widetilde Y_b]\E[\widetilde Y_b^{-1}]=\E[\widetilde Y_b/\widetilde Y_b']$, and by symmetry the same holds with $\widetilde Y_b$ and $\widetilde Y_b'$ swapped. Thus,
\begin{align}
V_\alpha(b)&=\E[\widetilde Y_b]\E[\widetilde Y_b^{-1}]-1\nonumber\\
&=\frac12\E\!\left[\frac{\widetilde Y_b}{\widetilde Y_b'}+\frac{\widetilde Y_b'}{\widetilde Y_b}-2\right]\nonumber\\
&=\frac12\E\!\left[\frac{(\widetilde Y_b-\widetilde Y_b')^2}{\widetilde Y_b\widetilde Y_b'}\right]\nonumber\\
&=2\E\bigl[G(\widetilde Y_b,\widetilde Y_b')\bigr].\label{eq:independent-nodes}
\end{align}
Because $\widetilde Y_b=x_I$ and $\widetilde Y_b'=x_{I'}$, the spacing of the nodes will allow us to bound $\E[G(\widetilde Y_b,\widetilde Y_b')]$ from below in terms of $\Var(I)$. In the next lemma we show that the cancellation condition \eqref{eq:cancellation} forces $\Var(I)\geq1/4$. Proposition~\ref{prop:node-bound} uses this variance bound to identify all minimizing vectors.

\begin{lemma}\label{lem:index-variance}
Let $\boldsymbol{s}=(s_1,\ldots,s_n)\in\R^n$ be nonzero, with $\sum_{i=1}^n s_i=0$. Put $S_s=\sum_{i=1}^n s_i^2$ and define a random index $I$ by $\mathbb P(I=i)=s_i^2/S_s$. Then $\Var(I)\geq1/4$, with equality precisely when $\boldsymbol{s}$ has exactly two nonzero entries, at adjacent indices, with opposite values.
\end{lemma}

\begin{proof}
Put $\mu=\E[I]$, and for $0\leq i\leq n$ define
\[
c_0=0,\qquad c_i=\sum_{j=1}^i s_j,\qquad d_i=i-\mu+\frac12,\qquad a_i=d_ic_i.
\]
Because $\sum_{i=1}^n s_i=0$, we have $c_0=c_n=0$. Moreover, $d_i-d_{i-1}=1$ and
\[
(i-\mu)^2-\frac14=\left(i-\mu+\frac12\right)\left(i-\mu-\frac12\right)=d_id_{i-1}.
\]
Using $s_i=c_i-c_{i-1}$ and $a_i=d_ic_i$, we obtain
\begin{align*}
(a_i-a_{i-1})^2-d_id_{i-1}s_i^2&=(d_ic_i-d_{i-1}c_{i-1})^2-d_id_{i-1}(c_i-c_{i-1})^2\\
&=d_i(d_i-d_{i-1})c_i^2+d_{i-1}(d_{i-1}-d_i)c_{i-1}^2\\
&=d_ic_i^2-d_{i-1}c_{i-1}^2.
\end{align*}
Summing over $i=1,\ldots,n$ makes the right-hand side telescope to $d_nc_n^2-d_0c_0^2=0$. On the left, the sum of the subtracted terms is
\[
\sum_{i=1}^nd_id_{i-1}s_i^2=\sum_{i=1}^n\left((i-\mu)^2-\frac14\right)s_i^2=S_s\left(\Var(I)-\frac14\right).
\]
Rearranging and dividing by $S_s$ gives
\begin{equation}\label{eq:variance-squares}
\Var(I)=\frac14+\frac1{S_s}\sum_{i=1}^n(a_i-a_{i-1})^2\geq\frac14.
\end{equation}

Equality in \eqref{eq:variance-squares} holds if and only if $a_i=a_{i-1}$ for all $1\leq i\leq n$, so all $a_i$ are equal. Since $c_0=c_n=0$, the endpoint values vanish ($a_0=a_n=0$), so $(i-\mu+1/2)c_i=0$ for every $i$. Because not all $s_i$ vanish and $s_i=c_i-c_{i-1}$, at least one partial sum $c_j$ with $1\leq j<n$ must be nonzero. For this $j$, the relation $(j-\mu+1/2)c_j=0$ forces $\mu=j+1/2$, and since $i-\mu+1/2=i-j\neq0$ for $i\neq j$, all other partial sums vanish. Recovering $s_i=c_i-c_{i-1}$ yields $s_j=c_j$, $s_{j+1}=-c_j$, and $s_i=0$ otherwise, which gives exactly the stated pair of opposite entries. Conversely, such a pair gives equal probabilities $1/2$ on two adjacent indices and variance $1/4$.
\end{proof}

We next bound $\E[\widetilde Y]\E[\widetilde Y^{-1}]-1$ from below by the smallest adjacent relative gap, for any strictly increasing sequence of positive nodes.

\begin{proposition}\label{prop:node-bound}
Let $0<x_1<\cdots<x_n$ and let the nonzero real vector $\boldsymbol{s}=(s_1,\ldots,s_n)$ satisfy $\sum_{r=1}^n s_r=0$. Let $\widetilde Y$ be a random variable with $\mathbb P(\widetilde Y=x_r)=s_r^2/\sum_{k=1}^n s_k^2$, and put
\[
\delta=\min_{1\leq j<n}\frac{(x_{j+1}-x_j)^2}{4x_jx_{j+1}}.
\]
Then
\begin{equation}\label{eq:node-bound}
\E[\widetilde Y]\E[\widetilde Y^{-1}]-1\geq\delta.
\end{equation}
Equality holds exactly when $\boldsymbol{s}$ has only two nonzero entries, at an adjacent pair $(j,j+1)$ attaining the minimum, with opposite values.
\end{proposition}

\begin{proof}
Logarithmic coordinates turn the relative gaps into additive gaps. Let
\[
\ell=\min_{1\leq j<n}\log\frac{x_{j+1}}{x_j}>0,
\]
so that $\delta=\sinh^2(\ell/2)>0$. For $r<k$, adding adjacent gaps gives
\[
\log\frac{x_k}{x_r}=\sum_{j=r}^{k-1}\log\frac{x_{j+1}}{x_j}\geq(k-r)\ell.
\]
Since $\sinh$ is increasing and convex on $[0,\infty)$ and vanishes at zero, $\sinh(du)\geq d\sinh u$ for $d\geq1$ and $u\geq0$. With $d=k-r$ and $u=\ell/2$, applying $\sinh$ to half the preceding bound on the logarithmic gap and squaring gives
\begin{equation}\label{eq:gap-comparison}
G(x_r,x_k)\geq(k-r)^2\delta.
\end{equation}
By symmetry of $G$ (and $G(x_r,x_r)=0$), \eqref{eq:gap-comparison} holds for all $1\leq r,k\leq n$. Let $I,I'$ be independent and identically distributed indices with $\mathbb P(I=r)=s_r^2/\sum_{k=1}^n s_k^2$, and put $\widetilde Y=x_I$ and $\widetilde Y'=x_{I'}$. Substituting $r=I$ and $k=I'$ into \eqref{eq:gap-comparison} gives $G(x_I,x_{I'})\geq\delta(I-I')^2$. Symmetrizing the ratio $x_I/x_{I'}$ as in \eqref{eq:independent-nodes}, then applying Lemma~\ref{lem:index-variance}, gives
\begin{align*}
\E[\widetilde Y]\E[\widetilde Y^{-1}]-1&=2\E\bigl[G(x_I,x_{I'})\bigr]\\
&\geq2\delta\E\bigl[(I-I')^2\bigr]\\
&=4\delta\Var(I)\geq\delta,
\end{align*}
where $\E[(I-I')^2]=2\Var(I)$ by independence.

Since $\delta>0$, equality in this chain requires both $\Var(I)=1/4$ and equality in \eqref{eq:gap-comparison} almost surely for $(I,I')$. By Lemma~\ref{lem:index-variance}, $\Var(I)=1/4$ forces $\boldsymbol{s}$ to have only two nonzero entries, at adjacent indices $j$ and $j+1$, with opposite values. Equality in \eqref{eq:gap-comparison} for $(r,k)=(j,j+1)$ then requires their relative gap $G(x_j,x_{j+1})$ to equal $\delta$. Conversely, such a pair assigns probability $1/2$ to each of $x_j$ and $x_{j+1}$ under the law of $\widetilde Y$ and attains \eqref{eq:node-bound}.
\end{proof}

Applied to the Laguerre zeros, the equality case of Proposition~\ref{prop:node-bound} will single out polynomials of the form $L_{m+2}^{(\alpha)}(t)/((t-x)(t-y))$, as shown in the proof of Theorem~\ref{thm:polynomial}. For such a polynomial, quadrature leaves only two nonzero masses, making its normalization and moments explicit. The next lemma records them for any pair of Laguerre zeros.

\begin{lemma}\label{lem:pair-moments}
Let $x<y$ be any two zeros of $L_{m+2}^{(\alpha)}$, and put
\[
b(t)=\frac{L_{m+2}^{(\alpha)}(t)}{(t-x)(t-y)}.
\]
Then
\begin{equation}\label{eq:pair-moments}
\begin{aligned}
M_0^{(\alpha)}(b)&=\frac{C_{m+2,\alpha}(x+y)}{xy(y-x)^2},\\
\E\bigl[T_b^{(\alpha)}\bigr]&=\frac{2xy}{x+y},\qquad \E\bigl[(T_b^{(\alpha)})^2\bigr]=xy,\qquad V_\alpha(b)=\frac{(y-x)^2}{4xy}.
\end{aligned}
\end{equation}
The first two moments of $T_b^{(\alpha)}$ agree with those of the two-point law having mass $y/(x+y)$ at $x$ and mass $x/(x+y)$ at $y$.
\end{lemma}

\begin{proof}
For $x=x_i$ and $y=x_j$, multiplying the partial fraction identity
\[
\frac{1}{(t-x)(t-y)}=\frac{1}{x-y}\left(\frac{1}{t-x}-\frac{1}{t-y}\right)
\]
by $L_{m+2}^{(\alpha)}(t)$ identifies the two nonzero entries in \eqref{eq:interpolation} as $s_i=1/(x-y)$ and $s_j=1/(y-x)$. By \eqref{eq:discrete-moments},
\[
M_k^{(\alpha)}(b)=\frac{C_{m+2,\alpha}}{(y-x)^2}\bigl(x^{k-1}+y^{k-1}\bigr),\qquad k=0,1,2.
\]
Since $w_rb(x_r)^2=C_{m+2,\alpha}s_r^2/x_r$, the two-point law \eqref{eq:node-law} of $Y_b$ is supported on $\{x,y\}$ with masses proportional to $1/x$ and $1/y$, that is, $y/(x+y)$ at $x$ and $x/(x+y)$ at $y$. Dividing the cases $k=1,2$ by the case $k=0$ gives $\E[T_b^{(\alpha)}]=2xy/(x+y)$ and $\E[(T_b^{(\alpha)})^2]=xy$, matching the first two moments of $Y_b$, and substituting these into \eqref{eq:scv} gives the stated SCV.
\end{proof}

The equality condition on the nodes now determines the minimizing polynomials.

\begin{theorem}[Polynomial optimization]\label{thm:polynomial}
For $m\geq0$ and $\alpha\in\{0,1\}$,
\[
\min_{0\neq b\in\Pcal_m}V_\alpha(b)=\Delta_{m+2}^{(\alpha)}.
\]
The minimizing polynomials are precisely the nonzero multiples of the polynomials $b^*$ in \eqref{eq:optimizer-shape}, as $j^*$ ranges over the indices attaining $\Delta_{m+2}^{(\alpha)}$. Each such polynomial has exact degree $m$.
\end{theorem}

\begin{proof}
For every $0\neq b\in\Pcal_m$, interpolation gives \eqref{eq:cancellation}. Applying Proposition~\ref{prop:node-bound} to the Laguerre zeros $0<x_1<\cdots<x_{m+2}$, for which $\delta=\Delta_{m+2}^{(\alpha)}$, and combining with \eqref{eq:size-bias} gives $V_\alpha(b)\geq\Delta_{m+2}^{(\alpha)}$. Equality holds precisely when the vector in \eqref{eq:coordinates} has only two nonzero entries, at a minimizing adjacent pair $(j^*,j^*+1)$, with opposite values. Writing $x=x_{j^*}<y=x_{j^*+1}$ for that pair (with $s_{j^*}=c$ and $s_{j^*+1}=-c$), interpolation \eqref{eq:interpolation} then reads
\[
b(t)=cL_{m+2}^{(\alpha)}(t)\left(\frac1{t-x}-\frac1{t-y}\right)=c(x-y)\frac{L_{m+2}^{(\alpha)}(t)}{(t-x)(t-y)},\qquad c\neq0,
\]
which is a nonzero multiple of \eqref{eq:optimizer-shape}. Conversely, by Lemma~\ref{lem:pair-moments} and scale invariance of $V_\alpha$, every nonzero multiple of a polynomial in \eqref{eq:optimizer-shape} attains the bound $\Delta_{m+2}^{(\alpha)}$ and has exact degree $m$ (since $\deg L_{m+2}^{(\alpha)}=m+2$).
\end{proof}

Thus $\Delta_{m+2}^{(\alpha)}$ is the sharp minimum SCV among all single-pole polynomial-square densities of parity $\alpha$ and degree parameter $m$, with all minimizing shapes given by \eqref{eq:optimizer-shape}. In the next section we show that every real-spectrum ME density of order at most $N$ is a positive mixture of such single-pole densities with $2m_j+1+\alpha_j\leq N$, which reduces the full problem to comparing these single-pole bounds.

\section{Mixture reduction}\label{sec:reduction}

We now extend the single-pole bound of Section~\ref{sec:single-pole} to all real-spectrum ME densities of order at most $N$. The two mixtures in Theorem~\ref{thm:decomposition} play different roles. The finite mixture allows different polynomial factors, whereas each rate mixture varies the exponential rate while keeping the polynomial factor fixed. First, Section~\ref{sec:scv-mixtures} records an exact SCV identity for mixtures showing that any lower bound on the components passes to the mixture, with equality only when almost every component achieves the bound and the component means are almost surely constant. Next, we prove Theorem~\ref{thm:decomposition} by splitting a density into terminal components that have the maximal number of nonnegative zeros and then representing each terminal component as a positive rate mixture of polynomial-square densities. Finally, we combine these ingredients to prove Theorem~\ref{thm:optimizers} and deduce a strict monotonicity bound for the smallest relative gaps (Corollary~\ref{cor:monotonicity}).

\subsection{SCV of mixtures}\label{sec:scv-mixtures}

A lower bound on the SCV of every component also holds for the mixture. To see when equality can hold, we keep track of both the component SCVs and the variance of the component means.

\begin{lemma}\label{lem:mixture}
Let $\eta$ be a probability measure on a measurable space $\mathsf T$, and let $(g_\theta)_{\theta\in\mathsf T}$ be probability densities on $[0,\infty)$ with $(\theta,t)\mapsto g_\theta(t)$ measurable. Suppose each $g_\theta$ and the mixture $f(t)=\int_{\mathsf T}g_\theta(t)\,\eta(\di\theta)$ have finite second moments. Write $\mu_\theta=\int_0^\infty tg_\theta(t)\dd t>0$, $v_\theta=\SCV(g_\theta)$, and $v_*=\inf_{\theta\in\mathsf T}v_\theta$. For $\Theta\sim\eta$, let $\E_\eta$ and $\Var_\eta$ denote expectation and variance over $\Theta$. Then
\begin{equation}\label{eq:mixture-bound}
\SCV(f)-v_*=\frac{\E_\eta[(v_\Theta-v_*)\mu_\Theta^2]+(1+v_*)\Var_\eta(\mu_\Theta)}{(\E_\eta\mu_\Theta)^2}.
\end{equation}
In particular, $\SCV(f)\geq v_*$. Equality holds if and only if $v_\Theta=v_*$ and $\mu_\Theta$ is constant, $\eta$-almost surely.
\end{lemma}

\begin{proof}
Sample $\Theta\sim\eta$ and then $X$ with conditional density $g_\Theta$. The marginal density of $X$ is $f$, and conditioning gives
\[
\E[X]=\E_\eta\mu_\Theta,\qquad \E[X^2]=\E_\eta[(1+v_\Theta)\mu_\Theta^2].
\]
Subtracting $(1+v_*)(\E_\eta\mu_\Theta)^2$ from the second identity gives the numerator of \eqref{eq:mixture-bound}. Both its terms are nonnegative, and they vanish exactly under the stated conditions.
\end{proof}

To apply the equality criterion to Theorem~\ref{thm:decomposition}, we need to know when the means in a rate mixture can agree. Within each rate mixture, the polynomial factor is fixed. For a nonzero polynomial $\psi$ and $\Psi(t)=t^\alpha\psi(t)^2$, write $g_\lambda(t)=\Psi(t)\e^{-\lambda t}/z(\lambda)$ with $z(\lambda)=\int_0^\infty \Psi(t)\e^{-\lambda t}\dd t$, and write $\mu(\lambda)$ for the mean under $g_\lambda$. Differentiating $z(\lambda)$ brings down a factor $-t$, so $\mu(\lambda)=-z'(\lambda)/z(\lambda)$ and
\begin{equation}\label{eq:mean-derivative}
\mu'(\lambda)=-\left[\frac{z''(\lambda)}{z(\lambda)}-\left(\frac{z'(\lambda)}{z(\lambda)}\right)^2\right]=-\Var_{g_\lambda}(X)<0,\qquad X\sim g_\lambda.
\end{equation}
The exponential factor permits differentiation under the integral for every $\lambda>0$, and $g_\lambda$ is positive except at the finitely many zeros of $\Psi$. Consequently, any nondegenerate rate law $\Pi_j$ in \eqref{eq:decomposition} yields strictly varying conditional means, so the equality criterion of Lemma~\ref{lem:mixture} will force every rate law $\Pi_j$ of an optimizing density to collapse to a point mass.

\subsection{Splitting at zeros}

The proof of Theorem~\ref{thm:decomposition} starts with the zeros of the density. Lemma~\ref{lem:splitting} shows how a density with fewer than the maximal number of zeros in $[0,\infty)$, counted with multiplicity, can be split while preserving its zeros. At the maximal zero count, Lemma~\ref{lem:contacts} shows that the poles and zeros determine the density up to scale. Proposition~\ref{prop:rate-mixture} then constructs its positive mixture over rates.

To split a density $f$, we vary the coefficients in its exponential-polynomial expansion. The perturbation must preserve both the mass of $f$ and its zeros. We first describe the space of perturbations and count the conditions imposed by the zeros.

Let $Q$ be a monic polynomial of degree $n\geq1$ with real negative roots. In the splitting argument, $Q$ will be the reduced denominator of the Laplace transform of $f$. Write
\[
Q(z)=\prod_{i=1}^k(z+p_i)^{n_i},\qquad p_i>0,\qquad \sum_{i=1}^k n_i=n,
\]
where the $p_i$ are distinct. With $D=\di/\di t$, the exponential polynomials with these rates and degree bounds form the space
\[
\Vcal_Q=\{h:\R\to\R\mid Q(D)h=0\}=\left\{\sum_{i=1}^k P_i(t)\e^{-p_it}\;\middle|\;P_i\text{ real},\ \deg P_i<n_i\right\}.
\]
Indeed, the functions $\varphi_{i,r}(t)=t^r\e^{-p_it}$, $0\leq r<n_i$, are annihilated by $(D+p_i)^{n_i}$, since $(D+p_i)(P(t)\e^{-p_it})=P'(t)\e^{-p_it}$. The factors of $Q(D)$ commute, so each $\varphi_{i,r}$ solves $Q(D)h=0$. These $n$ functions are linearly independent and the differential equation has order $n$, so they form a basis.

This basis also shows that the reduced Laplace denominator of every member of $\Vcal_Q$ divides $Q$. For $h=\sum_{i,r}c_{i,r}\varphi_{i,r}$, its Laplace transform is
\[
\widehat h(s)=\int_0^\infty\e^{-st}h(t)\dd t=\sum_{i=1}^k\sum_{r=0}^{n_i-1}\frac{c_{i,r}r!}{(s+p_i)^{r+1}},\qquad s\geq0.
\]
Every denominator on the right divides $Q(s)$, so the sum can be written over $Q(s)$ and canceling common factors can only reduce that denominator. Thus a nonzero member of $\Vcal_Q$ that is nonnegative on $[0,\infty)$ is, after normalization, a real-spectrum ME density of order at most $n$.

For a nonzero function $f\in\Vcal_Q$ that is nonnegative on $[0,\infty)$, we call its zeros in $[0,\infty)$ \emph{contacts}, counting multiplicities and including a zero at the origin if one is present. Their number is at most $n-1$ (Lemma~\ref{lem:zero-count}); this follows from Rolle's theorem and is the extended Chebyshev property of real exponential polynomials \cite[Ch.~I, Sec.~3, Example~1]{karlinstudden1966}.

\begin{lemma}\label{lem:zero-count}
Every nonzero function in $\Vcal_Q$ has at most $n-1$ real zeros, counted with multiplicity.
\end{lemma}

\begin{proof}
Induct on the degree $n$ of $Q$. For $n=1$, every nonzero member is a constant times one exponential and has no zeros. Suppose $n>1$ and assume the assertion holds for every monic polynomial of degree $n-1$ with real negative roots. Let $h\in\Vcal_Q$ be nonzero and choose a root $-p$ of $Q$. Applying $D+p$ removes one factor from the differential equation. With $h_1=(D+p)h$ and $Q_1(z)=Q(z)/(z+p)$, we have $Q_1(D)h_1=Q(D)h=0$. Thus $h_1\in\Vcal_{Q_1}$, whose dimension is $n-1$. If $h_1=0$, then $h(t)=c\e^{-pt}$ has no zeros. Otherwise the induction hypothesis bounds the zeros of $h_1$ by $n-2$.

To transfer this bound to $h$, put $g(t)=\e^{pt}h(t)$, so $g'(t)=\e^{pt}h_1(t)$. Multiplication by an exponential changes neither the zeros nor their multiplicities. If $g$ has at least one real zero, take any finite collection $t_1<\cdots<t_q$ ($q\geq1$) of its zeros, with multiplicities $r_1,\ldots,r_q$. At $t_i$, the derivative $g'$ has multiplicity at least $r_i-1$, and Rolle's theorem gives a further zero in each interval $(t_i,t_{i+1})$. Hence $g'$ has at least
\[
\sum_{i=1}^q(r_i-1)+(q-1)=\sum_{i=1}^q r_i-1
\]
zeros counted with multiplicity. Since $g'$ and $h_1$ have the same zeros, $\sum_{i=1}^q r_i-1\leq n-2$, proving the bound.
\end{proof}

Call a density of exact order $n$ \emph{terminal} if it has $n-1$ contacts. To reach such densities, we perturb $f$ in two directions until each perturbation either gains a contact or loses a degree in its reduced Laplace denominator. The next lemma justifies this splitting procedure.

\begin{lemma}\label{lem:splitting}
Let $f$ be a real-spectrum ME density with monic reduced Laplace denominator $Q$. There exist an integer $J\geq1$, positive weights $\omega_1,\ldots,\omega_J$ summing to one, and real-spectrum ME densities $f_1,\ldots,f_J$ such that $f=\sum_{j=1}^J\omega_jf_j$. For each $j$, the reduced Laplace denominator $Q_j$ of $f_j$ divides $Q$, and $f_j$ is terminal.
\end{lemma}

\begin{proof}
Write $n=\deg Q$ and let $M$ be the total multiplicity of the contacts of $f$. By Lemma~\ref{lem:zero-count}, $M\leq n-1$. If $M=n-1$, the required decomposition consists of $f$ itself with weight one. Suppose therefore that $M<n-1$.

We first choose a nonzero perturbation direction $h\in\Vcal_Q$ such that $f+\tau h$ remains a probability density for all $\tau\in\R$ sufficiently close to zero. Preserving total mass one requires $\int_0^\infty h(t)\dd t=0$. Preserving nonnegativity for small $\tau$ of either sign requires $h$ to vanish at every contact of $f$ with at least the same multiplicity; otherwise $|h|/f$ would be unbounded near that contact, making $f+\tau h$ negative there for arbitrarily small $\tau$ of the appropriate sign.

A contact $\zeta$ of multiplicity $r$ imposes the $r$ linear conditions $h(\zeta)=h'(\zeta)=\cdots=h^{(r-1)}(\zeta)=0$. The $M$ contact conditions have rank at most $M$, so their common solution space has dimension at least $n-M\geq2$. Integration is a nonzero linear functional on this space, since the space contains $f$ and $\int_0^\infty f(t)\dd t=1$. Requiring zero integral therefore removes exactly one dimension, leaving a nonzero choice of $h$. This $h$ cannot be proportional to $f$, because their integrals are zero and one, respectively.

Having fixed $h$, we determine the largest admissible step sizes from the extrema of $h/f$. Because $h$ vanishes at every contact $\zeta$ of $f$ to at least the same multiplicity $r$, the quotient $h/f$ extends continuously across $\zeta$ with value $h^{(r)}(\zeta)/f^{(r)}(\zeta)$. To compute its limit at infinity, let $p_0>0$ be the smallest rate with $Q(-p_0)=0$ and let $d+1$ be the multiplicity of $-p_0$ in the reduced denominator $Q$. Then the polynomial multiplying $\e^{-p_0t}$ in $f$ has exact degree $d$; since $f\geq0$ and all other rates exceed $p_0$, we have $f(t)\sim ct^d\e^{-p_0t}$ for some $c>0$. Because $h\in\Vcal_Q$, its coefficient of $t^d\e^{-p_0t}$ divided by $c$ gives $\lim_{t\to\infty}h(t)/f(t)$. Taking the extrema over $[0,\infty]$, put
\[
\ell=\min\frac{h}{f}<0<\max\frac{h}{f}=u.
\]
Both signs occur because $h\not\equiv0$ and $\int_0^\infty h(t)\dd t=0$. Since $f+\tau h=f(1+\tau h/f)$ away from the contacts, it is nonnegative exactly for $-1/u\leq\tau\leq-1/\ell$. The endpoint densities
\begin{equation}\label{eq:split-densities}
f_1=f-\frac{h}{u},\qquad f_2=f-\frac{h}{\ell}
\end{equation}
are nonnegative and each integrates to one. They belong to $\Vcal_Q$ and satisfy
\begin{equation}\label{eq:binary-mixture}
f=\frac{u}{u-\ell}f_1+\frac{-\ell}{u-\ell}f_2.
\end{equation}

Both $f_1$ and $f_2$ vanish at every contact of $f$ with at least the same multiplicity. For $f_1$, consider the maximum $u$ of $h/f$, and for $f_2$, its minimum $\ell$. If this extremum is attained at a finite point $\zeta\geq0$, then $1-(h/f)/u$ (respectively, $1-(h/f)/\ell$) vanishes at $\zeta$ (canceling the leading term $(t-\zeta)^r$ if $\zeta$ is already a contact of $f$ of multiplicity $r$) and thus creates a new contact or strictly increases an existing multiplicity. If the extremum is attained only at infinity, the coefficient of $t^d\e^{-p_0t}$ is canceled, so the reduced Laplace denominator of that component loses at least one degree. Thus, for each new component, the quantity
\[
\text{exact order}-\text{total multiplicity of contacts}
\]
strictly decreases. This integer is always at least one by Lemma~\ref{lem:zero-count}. Repeating the split in the space associated with each new density's reduced Laplace denominator therefore stops after finitely many steps. The reduced Laplace denominator of every component divides that of the density from which it was split, and the final components are terminal.
\end{proof}

\begin{example}\label{ex:splitting}
Put $\xi=\log2$ and consider the five-pole density
\begin{align*}
f(t)&=\frac{15}{272}\e^{-t}(1-2\e^{-t})^2\bigl(281\e^{-2t}-166\e^{-t}+25\bigr)\\
&=\frac{15}{272}\bigl(25\e^{-t}-266\e^{-2t}+1045\e^{-3t}-1788\e^{-4t}+1124\e^{-5t}\bigr),\qquad t\geq0.
\end{align*}
Writing $281\e^{-2t}-166\e^{-t}+25=17(1-3\e^{-t})^2+8(1-4\e^{-t})^2$ shows that the quadratic factor is strictly positive, and termwise integration confirms that $f$ has mass one. All five coefficients in the expansion are nonzero, so $f$ has order five with $Q(s)=\prod_{i=1}^5(s+i)$, and the factorization gives a single double contact at $\xi$. To split $f$, choose the perturbation
\[
h(t)=\frac{15}{272}\e^{-t}(1-2\e^{-t})^2\bigl(8(1-4\e^{-t})^2-17(1-3\e^{-t})^2\bigr)\in\Vcal_Q,
\]
which vanishes to order two at $\xi$ and has zero integral by termwise integration. Comparing the two square terms shows that $|h|\leq f$, with $h=f$ at $\zeta_1=\log3$ and $h=-f$ at $\zeta_2=\log4$. Thus $\ell=-1$, $u=1$, and \eqref{eq:split-densities}--\eqref{eq:binary-mixture} yield $f=\frac12f_1+\frac12f_2$ with
\begin{align*}
f_1(t)&=f(t)-h(t)=\frac{15}{8}\e^{-t}(1-2\e^{-t})^2(1-3\e^{-t})^2,\\
f_2(t)&=f(t)+h(t)=\frac{15}{17}\e^{-t}(1-2\e^{-t})^2(1-4\e^{-t})^2.
\end{align*}
Both components retain all five poles and the double contact at $\xi$, while $f_1$ and $f_2$ gain double contacts at $\zeta_1$ and $\zeta_2$, respectively. Each now has order five and four contacts, completing the decomposition.
\end{example}

The upper bound of $n-1$ zeros in Lemma~\ref{lem:zero-count} also tells us how to recognize a terminal density. To prove that two nonzero functions in $\Vcal_Q$ with the same $n-1$ prescribed contacts are proportional, it suffices to scale one to match the other at a single point outside the contacts. This is made precise in the following lemma.

\begin{lemma}\label{lem:contacts}
Let $f\in\Vcal_Q$ be nonzero and nonnegative on $[0,\infty)$, with reduced Laplace denominator $Q$ of degree $n$. Suppose that $f$ has $n-1$ contacts, listed as $\zeta_1,\ldots,\zeta_{n-1}$ with multiplicity. Put
\begin{equation}\label{eq:contact-polynomial}
\Psi(t)=\prod_{j=1}^{n-1}(t-\zeta_j),
\end{equation}
with the empty product equal to one when $n=1$. We call $\Psi$ the \emph{contact polynomial} of $f$. Every $h\in\Vcal_Q$ satisfying $h^{(k)}(\zeta)=0$ for $0\leq k<r$ at each contact $\zeta$ of $f$ of multiplicity $r$ is a multiple of $f$. Moreover, $\Psi\geq0$ on $[0,\infty)$ and has the form $\Psi(t)=t^\alpha\psi(t)^2$, where $n=2m+1+\alpha$, $\alpha\in\{0,1\}$, and $\psi$ has degree $m$ with all its zeros nonnegative.
\end{lemma}

\begin{proof}
Choose $t_0\geq0$ with $f(t_0)>0$ and put
\[
g(t)=h(t)-\frac{h(t_0)}{f(t_0)}f(t).
\]
Since $\Vcal_Q$ is a linear space, $g\in\Vcal_Q$. The derivative conditions on $h$ and $f$ give $g$ all the prescribed zeros, with total multiplicity at least $n-1$, and the choice of coefficient gives $g(t_0)=0$. This is an additional zero because $f(t_0)>0$. If $g$ were nonzero, it would therefore have at least $n$ zeros, contradicting Lemma~\ref{lem:zero-count}. Thus $g\equiv0$, which gives $h=h(t_0)f/f(t_0)$.

Every positive contact of $f$ has even multiplicity, since an odd multiplicity would make $f$ change sign there. Let $r_0$ be the multiplicity at the origin, taking $r_0=0$ when $f(0)>0$. The positive contacts contribute an even number to the total $n-1$, so $r_0$ and $n-1$ are either both even or both odd. Set $\alpha=0$ in the even case and $\alpha=1$ in the odd case. After removing $t^\alpha$ from $\Psi$, all remaining root multiplicities are even. Taking half of each remaining multiplicity defines a monic polynomial $\psi$ with $\Psi(t)=t^\alpha\psi(t)^2$, degree $m=(n-1-\alpha)/2$, and only nonnegative zeros. This factorization also gives $\Psi\geq0$ on $[0,\infty)$.
\end{proof}

\subsection{Positive rate mixtures}

Fix a terminal density $f$ with reduced Laplace denominator $Q$ of degree $n$ and contact polynomial $\Psi$ from \eqref{eq:contact-polynomial}. We prove that the quotient $f/\Psi$, continued across the contacts, is a positive mixture of exponentials. Normalizing the kernels then gives the probability mixture required for each terminal component in Theorem~\ref{thm:decomposition}.

To see why only the exponential rate can vary in a positive mixture of single-pole ME densities of order at most $n$, recall that each single-pole component has the form $P(t)\e^{-\lambda t}$, where $\lambda>0$ and $P\geq0$ is a nonzero polynomial of degree at most $n-1$. Positivity forces almost every component to share each contact of $f$ with at least its multiplicity, as any positive weight of components vanishing to lower order at a contact would prevent $f$ from vanishing to that order there. Since the contacts of $f$ have total multiplicity $n-1$ while $\deg P\leq n-1$, the polynomial factor $P$ is forced to be a positive multiple of $\Psi(t)=t^\alpha\psi(t)^2$, leaving only the rate $\lambda$ free to vary.

If $f$ has several poles, a finite sum of these kernels cannot suffice. With $q$ distinct rates and positive coefficients, the Laplace transform of $\sum_{i=1}^q c_i\Psi(t)\e^{-\lambda_i t}$ has a pole of multiplicity $n$ at each $-\lambda_i$, because $\deg \Psi=n-1$ and the transforms of the other summands are analytic there. Its order is therefore $qn$, so for a terminal density of order $n$ a finite sum would force $q=1$. A continuous rate mixture is therefore needed in the multi-pole case.

\begin{example}\label{ex:hypoexponential}
Suppose $n\geq2$ and all $n-1$ contacts of $f$ lie at the origin, so that $\Psi(t)=t^{n-1}$. Let $S=\sum_{i=1}^n E_i/p_i$, where $E_1,\ldots,E_n$ are independent mean-one exponential random variables and $-p_1,\ldots,-p_n$ are the roots of $Q$ listed with multiplicity. Its density is given for $t>0$ by the $n$-fold convolution integral
\begin{align*}
f_S(t)&=\left(\prod_{i=1}^n p_i\right)\int_{\substack{x_1,\ldots,x_{n-1}\geq0\\x_1+\cdots+x_{n-1}\leq t}}\exp\!\left(-\sum_{i=1}^n p_ix_i\right)\dd x_1\cdots\dd x_{n-1}\\
&=\left(\prod_{i=1}^n p_i\right)t^{n-1}\int_{\substack{u_1,\ldots,u_{n-1}\geq0\\u_1+\cdots+u_{n-1}\leq1}}\exp\!\left(-t\sum_{i=1}^n p_iu_i\right)\dd u_1\cdots\dd u_{n-1}\\
&=\left(\prod_{i=1}^n p_i\right)\frac{t^{n-1}}{(n-1)!}\,\E\bigl[\e^{-\Lambda t}\bigr],\qquad \Lambda=\sum_{i=1}^n p_i U_i,
\end{align*}
where $x_n=t-\sum_{i=1}^{n-1}x_i$, the second equality substitutes $x_i=tu_i$ ($1\leq i\leq n$), and $(U_1,\ldots,U_n)$ is uniformly distributed on the probability simplex in $\R^n$. Since $\E[\e^{-sS}]=(\prod_{i=1}^n p_i)/Q(s)$, we have $f_S\in\Vcal_Q$ with $n-1$ contacts at the origin, so Lemma~\ref{lem:contacts} forces $f=f_S$. Letting $\nu$ denote the law of $\Lambda$, we may rewrite this as
\[
f(t)=\left(\prod_{i=1}^n p_i\right)\frac{t^{n-1}}{(n-1)!}\int\e^{-\lambda t}\,\nu(\di\lambda)=\int\frac{\lambda^nt^{n-1}\e^{-\lambda t}}{(n-1)!}\,\Pi(\di\lambda),
\]
where
\[
\Pi(\di\lambda)=\frac{\prod_{i=1}^n p_i}{\lambda^n}\,\nu(\di\lambda).
\]
Integrating in $t$ shows that $\Pi$ is a probability measure, so $f$ is a mixture of $n$-phase Erlang densities with rate law $\Pi$. Both $\nu$ and $\Pi$ are supported on $[\min_i p_i,\max_i p_i]$.
\end{example}

To extend Example~\ref{ex:hypoexponential} to arbitrary contact locations, we seek a probability measure $\nu$ supported between the smallest and largest rates of $Q$ such that
\[
\Phi(t)=\Psi(t)\int\e^{-\lambda t}\,\nu(\di\lambda)
\]
belongs to $\Vcal_Q$. A general choice of $\nu$ need not produce a rational Laplace transform. Once $\Phi\in\Vcal_Q$ is established, however, Lemma~\ref{lem:contacts} identifies $\Phi$ with a positive multiple of $f$. It therefore suffices to construct $\nu$ such that $Q(D)\Phi=0$.

For the auxiliary construction, take distinct rates $0<p_1<\cdots<p_n$ and distinct prescribed real zeros $\zeta_1<\cdots<\zeta_{n-1}$. Here $Q(z)=\prod_{i=1}^n(z+p_i)$, so $\Vcal_Q$ is the linear span of $\e^{-p_1t},\ldots,\e^{-p_nt}$; a limit will later recover the repeated rates and contact multiplicities of $f$. The cases $n=2$ and $n=3$ show how the rate gaps enter the construction. For $n=2$, a combination of $\e^{-p_1t}$ and $\e^{-p_2t}$ vanishing at $\zeta_1$ is proportional to
\[
\e^{-p_1(t-\zeta_1)}-\e^{-p_2(t-\zeta_1)}=(t-\zeta_1)\int_{p_1}^{p_2}\e^{\lambda\zeta_1}\e^{-\lambda t}\dd\lambda.
\]
After normalization, the mixture is taken against the uniform law on $(p_1,p_2)$ exponentially tilted by $\e^{\lambda\zeta_1}$. Moving the zero away from the origin therefore corresponds to exponentially tilting the rate measure in the $n=2$ case of Example~\ref{ex:hypoexponential}.

For $n=3$, the space $\Vcal_Q=\operatorname{span}\{\e^{-p_1t},\e^{-p_2t},\e^{-p_3t}\}$ is three-dimensional, and applying the $n=2$ construction to the adjacent rate pairs $(p_1,p_2)$ and $(p_2,p_3)$ yields two linearly independent functions $g_i(t)=\e^{-p_i(t-\zeta_1)}-\e^{-p_{i+1}(t-\zeta_1)}$ ($i=1,2$) in $\Vcal_Q$ that vanish at $\zeta_1$. To impose the second zero at $\zeta_2$, form the combination $g_1(\zeta_2)g_2(t)-g_2(\zeta_2)g_1(t)$. Multiplying row $i$ of the $3\times3$ determinant below by $\e^{p_i\zeta_1}$, subtracting row $2$ from row $1$ and then row $3$ from row $2$, and expanding along the first column shows that
\[
g_1(\zeta_2)g_2(t)-g_2(\zeta_2)g_1(t)=\e^{(p_1+p_2+p_3)\zeta_1}\det\begin{pmatrix}
\e^{-p_1\zeta_1}&\e^{-p_1\zeta_2}&\e^{-p_1t}\\
\e^{-p_2\zeta_1}&\e^{-p_2\zeta_2}&\e^{-p_2t}\\
\e^{-p_3\zeta_1}&\e^{-p_3\zeta_2}&\e^{-p_3t}
\end{pmatrix}.
\]
If instead we factor $\e^{-p_it}$ from row $i$ of this determinant and subtract each row from the one below it (working from the bottom row upward), then expanding along the last column, applying $\e^{p_{i+1}(t-\zeta_j)}-\e^{p_i(t-\zeta_j)}=(t-\zeta_j)\int_{p_i}^{p_{i+1}}\e^{\lambda_i(t-\zeta_j)}\dd\lambda_i$ to each entry of the remaining $2\times2$ minor, and using multilinearity in the rows yields
\begin{align*}
&\det\begin{pmatrix}
\e^{-p_1\zeta_1}&\e^{-p_1\zeta_2}&\e^{-p_1t}\\
\e^{-p_2\zeta_1}&\e^{-p_2\zeta_2}&\e^{-p_2t}\\
\e^{-p_3\zeta_1}&\e^{-p_3\zeta_2}&\e^{-p_3t}
\end{pmatrix}\\
&\qquad=\e^{-(p_1+p_2+p_3)t}\det\begin{pmatrix}
\e^{p_1(t-\zeta_1)}&\e^{p_1(t-\zeta_2)}&1\\
\e^{p_2(t-\zeta_1)}-\e^{p_1(t-\zeta_1)}&\e^{p_2(t-\zeta_2)}-\e^{p_1(t-\zeta_2)}&0\\
\e^{p_3(t-\zeta_1)}-\e^{p_2(t-\zeta_1)}&\e^{p_3(t-\zeta_2)}-\e^{p_2(t-\zeta_2)}&0
\end{pmatrix}\\
&\qquad=\e^{-(p_1+p_2+p_3)t}(t-\zeta_1)(t-\zeta_2)\int_{p_1}^{p_2}\int_{p_2}^{p_3}\det\begin{pmatrix}
\e^{\lambda_1(t-\zeta_1)}&\e^{\lambda_1(t-\zeta_2)}\\
\e^{\lambda_2(t-\zeta_1)}&\e^{\lambda_2(t-\zeta_2)}
\end{pmatrix}\dd\lambda_2\dd\lambda_1\\
&\qquad=(t-\zeta_1)(t-\zeta_2)\int_{p_1}^{p_2}\int_{p_2}^{p_3}\det\begin{pmatrix}
\e^{-\lambda_1\zeta_1}&\e^{-\lambda_1\zeta_2}\\
\e^{-\lambda_2\zeta_1}&\e^{-\lambda_2\zeta_2}
\end{pmatrix}\e^{-(p_1+p_2+p_3-\lambda_1-\lambda_2)t}\dd\lambda_2\dd\lambda_1,
\end{align*}
where the last equality factors $\e^{\lambda_1 t}$ and $\e^{\lambda_2 t}$ out of the two rows of the $2\times2$ determinant. For general $n$, the same row-reduction procedure yields an integral over the $n-1$ rate gaps with an exponential determinant giving the joint weight. The following lemma makes this representation precise.

\begin{lemma}\label{lem:rate-gaps}
Let $n\geq2$, $0<p_1<\cdots<p_n$, and $\zeta_1<\cdots<\zeta_{n-1}$. On the open box $B_{\boldsymbol{p}}=\prod_{i=1}^{n-1}(p_i,p_{i+1})$, write $\boldsymbol\lambda=(\lambda_1,\ldots,\lambda_{n-1})$ and define
\begin{equation}\label{eq:rate-weight}
\begin{aligned}
A(\boldsymbol\lambda)&=\sum_{i=1}^n p_i-\sum_{i=1}^{n-1}\lambda_i=p_1+\sum_{i=1}^{n-1}(p_{i+1}-\lambda_i),\\
H(\boldsymbol\lambda)&=\left|\det\bigl(\e^{-\lambda_i\zeta_j}\bigr)_{i,j=1}^{n-1}\right|.
\end{aligned}
\end{equation}
Then $A(\boldsymbol\lambda)\in(p_1,p_n)$ and $H(\boldsymbol\lambda)>0$ for every $\boldsymbol\lambda\in B_{\boldsymbol{p}}$. With $\di\boldsymbol\lambda=\di\lambda_1\cdots\di\lambda_{n-1}$, put $z_H=\int_{B_{\boldsymbol{p}}}H(\boldsymbol\lambda)\dd\boldsymbol\lambda$. We have $0<z_H<\infty$, and
\begin{equation}\label{eq:rate-law}
\nu(E)=\frac1{z_H}\int_{B_{\boldsymbol{p}}}H(\boldsymbol\lambda)\mathbf1_{\{A(\boldsymbol\lambda)\in E\}}\dd\boldsymbol\lambda,\qquad E\subseteq\R\text{ Borel},
\end{equation}
defines a probability measure supported on $[p_1,p_n]$. If $\Psi_{\boldsymbol{\zeta}}(t)=\prod_{j=1}^{n-1}(t-\zeta_j)$ and $Q(z)=\prod_{i=1}^n(z+p_i)$, then
\begin{equation}\label{eq:rate-solution}
\Phi(t)=\Psi_{\boldsymbol{\zeta}}(t)\int\e^{-\lambda t}\,\nu(\di\lambda)
\end{equation}
belongs to $\Vcal_Q$.
\end{lemma}

\begin{proof}
For $\boldsymbol\lambda\in B_{\boldsymbol{p}}$, we have $p_i<\lambda_i<p_{i+1}$, so the coordinates $\lambda_1<\cdots<\lambda_{n-1}$ are strictly ordered. The determinant $\det(\e^{-\lambda_i\zeta_j})_{i,j=1}^{n-1}$ cannot vanish, since otherwise its rows would be linearly dependent, giving a nonzero linear combination of the $n-1$ functions $\e^{-\lambda_i t}$ that vanishes at all $n-1$ distinct points $\zeta_j$, contrary to Lemma~\ref{lem:zero-count}. By continuity on the connected box $B_{\boldsymbol{p}}$, this determinant has a constant sign there, and it is bounded on its closure. Hence $H(\boldsymbol\lambda)>0$ on $B_{\boldsymbol{p}}$ and $0<z_H<\infty$, so \eqref{eq:rate-law} defines a probability measure. Since $0<p_{i+1}-\lambda_i<p_{i+1}-p_i$ on $B_{\boldsymbol{p}}$, \eqref{eq:rate-weight} gives $A(\boldsymbol\lambda)\in(p_1,p_n)$, establishing the support assertion for $\nu$.

To show that $\Phi\in\Vcal_Q$, consider the $n\times n$ determinant
\begin{equation}\label{eq:determinant}
\Xi(t)=\det\begin{pmatrix}
\e^{-p_1\zeta_1}&\cdots&\e^{-p_1\zeta_{n-1}}&\e^{-p_1t}\\
\vdots&\ddots&\vdots&\vdots\\
\e^{-p_n\zeta_1}&\cdots&\e^{-p_n\zeta_{n-1}}&\e^{-p_nt}
\end{pmatrix}.
\end{equation}
Expanding along the last column expresses $\Xi(t)$ as a combination of $\e^{-p_1t},\ldots,\e^{-p_nt}$, so $\Xi\in\Vcal_Q$. At $t=\zeta_j$ the last column repeats column $j$. The row operations below extract these zeros and identify the measure left after division by $\Psi_{\boldsymbol{\zeta}}$.

To factor $\Psi_{\boldsymbol{\zeta}}(t)$ out of \eqref{eq:determinant}, factor $\e^{-p_it}$ from row $i$ for $1\leq i\leq n$, which pulls out $\e^{-t\sum_{i=1}^n p_i}$ and makes the last column equal to one. Working from the bottom row upward, subtract the preceding row from each row. The last column becomes $(1,0,\ldots,0)^{\mathsf T}$, while entry $(i+1,j)$ becomes
\[
\e^{p_{i+1}(t-\zeta_j)}-\e^{p_i(t-\zeta_j)}=(t-\zeta_j)\int_{p_i}^{p_{i+1}}\e^{\lambda_i(t-\zeta_j)}\dd\lambda_i.
\]
Passing to the $(n-1)\times(n-1)$ minor of entry $(1,n)$ introduces the sign $(-1)^{n-1}$. Factoring $t-\zeta_j$ from each remaining column, using multilinearity in the rows, and factoring $\e^{\lambda_it}$ from row $i$ yields
\[
\Xi(t)=(-1)^{n-1}\Psi_{\boldsymbol{\zeta}}(t)\int_{B_{\boldsymbol{p}}}\det\bigl(\e^{-\lambda_i\zeta_j}\bigr)_{i,j=1}^{n-1}\e^{-tA(\boldsymbol\lambda)}\dd\boldsymbol\lambda.
\]
The identity holds for every real $t$, including the points $t=\zeta_j$. The determinant inside the integral has a constant sign on $B_{\boldsymbol{p}}$, so replacing it by $H$ changes only the overall sign. Equations \eqref{eq:rate-law} and \eqref{eq:rate-solution} therefore make $\Phi$ a nonzero constant multiple of $\Xi$. Since $\Xi\in\Vcal_Q$, we conclude that $\Phi\in\Vcal_Q$.
\end{proof}

To interpret \eqref{eq:rate-law} probabilistically, let $L_1,\ldots,L_{n-1}$ be independent random variables with $L_i$ uniformly distributed on $(p_i,p_{i+1})$. Write $\boldsymbol L=(L_1,\ldots,L_{n-1})$. The measure $\nu$ is the law of $A(\boldsymbol L)$ after reweighting the joint law of $\boldsymbol L$ by $H(\boldsymbol L)/\E[H(\boldsymbol L)]$.

The exponential tilt observed for $n=2$ extends to the construction in Lemma~\ref{lem:rate-gaps}. Replacing each zero $\zeta_j$ by $\zeta_j+a$ for some $a\in\R$ multiplies the weight $H(\boldsymbol\lambda)$ by $\e^{-a\sum_{i=1}^{n-1}\lambda_i}$. By the definition of $A$ in \eqref{eq:rate-weight}, this factor is proportional to $\e^{aA(\boldsymbol\lambda)}$. Let $\nu_{\boldsymbol\zeta}$ and $\nu_{\boldsymbol\zeta+a}$ denote the rate measures associated with the original zero vector $\boldsymbol\zeta$ and the translated vector $\boldsymbol\zeta+a=(\zeta_1+a,\ldots,\zeta_{n-1}+a)$, respectively. Then
\[
\nu_{\boldsymbol\zeta+a}(\di\lambda)=\frac{\e^{a\lambda}}{\int\e^{au}\,\nu_{\boldsymbol\zeta}(\di u)}\,\nu_{\boldsymbol\zeta}(\di\lambda).
\]
Thus a common translation of the prescribed zeros corresponds to an exponential tilt of the rate measure.

\begin{remark}\label{rem:karlin-mcgregor}
Up to a Gaussian factor, the weight $H(\boldsymbol\lambda)$ in \eqref{eq:rate-weight} coincides with the unit-time transition density of $n-1$ independent standard Brownian motions killed upon first collision, starting at $\lambda_1<\cdots<\lambda_{n-1}$ and ending at $-\zeta_{n-1}<\cdots<-\zeta_1$ \cite{karlinmcgregor1959}. The restriction of each $\lambda_i$ to its corresponding rate gap ensures that the starting configuration is strictly ordered. Since the ending configuration is also strictly ordered, the transition density is strictly positive throughout $B_{\boldsymbol p}$.
\end{remark}

The construction so far requires distinct rates and distinct prescribed zeros. The latter restriction excludes, in particular, the double contacts of the terminal densities in Example~\ref{ex:splitting} as well as the repeated contact at the origin in Example~\ref{ex:hypoexponential} for $n\geq3$. Substituting a repeated contact directly would give identical columns in the determinant defining $H$, while a repeated rate would collapse one of the integration intervals. To handle both cases, we first normalize the rate measure for distinct parameters and then take a weak limit of these probability measures.

\begin{proposition}\label{prop:rate-mixture}
Let $f$ be a terminal real-spectrum ME density with reduced Laplace denominator $Q$ of degree $n$, roots $-p_1\geq\cdots\geq-p_n$ listed with multiplicity, and contact polynomial $\Psi$ from \eqref{eq:contact-polynomial}. There is a unique pair $(c,\nu)$ consisting of a constant $c>0$ and a probability measure $\nu$ supported on $[p_1,p_n]$ such that
\begin{equation}\label{eq:rate-mixture}
f(t)=c\,\Psi(t)\int\e^{-\lambda t}\,\nu(\di\lambda),\qquad t\geq0.
\end{equation}
If $p_1<p_n$, then $\nu$ is not a point mass.
\end{proposition}

\begin{proof}
For $n=1$, take $\nu=\delta_{p_1}$. For $n\geq2$, order the contacts $\zeta_1\leq\cdots\leq\zeta_{n-1}$ with multiplicity and choose approximations $p_i^\varepsilon\to p_i$ and $\zeta_j^\varepsilon\to\zeta_j$ as $\varepsilon\downarrow0$ such that $0<p_1^\varepsilon<\cdots<p_n^\varepsilon$ and $\zeta_1^\varepsilon<\cdots<\zeta_{n-1}^\varepsilon$ for each $\varepsilon$. Let $\nu_\varepsilon$ and $\Phi_\varepsilon$ be the probability measure and function supplied by Lemma~\ref{lem:rate-gaps} for these perturbed parameters, with corresponding polynomials $Q_\varepsilon$ and $\Psi_\varepsilon$.

Since the supports $[p_1^\varepsilon,p_n^\varepsilon]$ lie in a common compact interval, Prohorov's theorem yields a subsequence along which $\nu_\varepsilon$ converges weakly to a probability measure $\nu$ supported on $[p_1,p_n]$. For each fixed $t\in\R$ and integer $r\geq0$, the function $\lambda\mapsto(-\lambda)^r\e^{-\lambda t}$ is continuous and bounded on that compact interval. Thus the exponential integrals and all their derivatives in $t$ converge. Since $\Psi_\varepsilon\to \Psi$ coefficientwise, the product rule gives $\Phi_\varepsilon^{(r)}(t)\to\Phi^{(r)}(t)$, where $\Phi(t)=\Psi(t)\int\e^{-\lambda t}\,\nu(\di\lambda)$. Together with $Q_\varepsilon\to Q$ coefficientwise and $Q_\varepsilon(D)\Phi_\varepsilon=0$, this yields $Q(D)\Phi=0$, so $\Phi\in\Vcal_Q$. Since the integral is strictly positive, $\Phi$ is nonzero, nonnegative on $[0,\infty)$, and shares all $n-1$ contacts of $f$ with multiplicity. Lemma~\ref{lem:contacts} therefore gives $f=c\Phi$ for some $c>0$.

For uniqueness, evaluate the quotient $f/\Psi$, continued across the contacts, at zero to determine $c$. Uniqueness of Laplace transforms then determines $\nu$.

Finally, if $\nu=\delta_\lambda$, then $f(t)=c\Psi(t)\e^{-\lambda t}$. Since $\Psi$ is monic of degree $n-1$, this density has reduced Laplace denominator $(s+\lambda)^n$, so $p_1=p_n=\lambda$.
\end{proof}

A consequence of Proposition~\ref{prop:rate-mixture} is that the quotient $f/\Psi$, continued across the contacts, is completely monotone on $[0,\infty)$.

To turn \eqref{eq:rate-mixture} into a probability mixture, normalize each kernel and absorb its mass into the rate law. With $z(\lambda)$ and $g_\lambda$ as in Section~\ref{sec:scv-mixtures}, now for the contact polynomial $\Psi$ of $f$, integrating \eqref{eq:rate-mixture} in $t$ and applying Tonelli's theorem gives
\begin{equation}\label{eq:normalized-rate}
\Pi(\di\lambda)=cz(\lambda)\nu(\di\lambda),\qquad \Pi(\R)=1,\qquad f(t)=\int g_\lambda(t)\,\Pi(\di\lambda).
\end{equation}
We can now assemble the splitting and rate-mixture stages to prove Theorem~\ref{thm:decomposition}.

\begin{proof}[Proof of Theorem~\ref{thm:decomposition}]
Let $Q$ be the reduced denominator of the Laplace transform of $f$, so $\deg Q\leq N$. By Lemma~\ref{lem:splitting}, $f=\sum_{j=1}^J\omega_jf_j$ with $\omega_j>0$, $\sum_{j=1}^J\omega_j=1$, where each $f_j$ has reduced Laplace denominator $Q_j$ dividing $Q$ and exact order $r_j=\deg Q_j\leq N$, with $r_j-1$ contacts. Write $r_j=2m_j+1+\alpha_j$ with $m_j\geq0$ and $\alpha_j\in\{0,1\}$. By Lemma~\ref{lem:contacts}, the contact polynomial of $f_j$ has the form $\Psi_j(t)=t^{\alpha_j}\psi_j(t)^2$, where $\deg \psi_j=m_j$ and all roots of $\psi_j$ lie in $[0,\infty)$.

Applying Proposition~\ref{prop:rate-mixture} and \eqref{eq:normalized-rate} to each $f_j$ gives
\[
f_j(t)=\int_0^\infty\frac{\Psi_j(t)\e^{-\lambda t}}{z_j(\lambda)}\,\Pi_j(\di\lambda),\qquad z_j(\lambda)=\int_0^\infty \Psi_j(t)\e^{-\lambda t}\dd t,
\]
where $\Pi_j$ is a probability measure supported on the closed interval between the smallest and largest rates of $Q_j$. Write $\nu_j$ and $c_j$ for the measure and the constant that Proposition~\ref{prop:rate-mixture} attaches to $f_j$, so that $\Pi_j(\di\lambda)=c_jz_j(\lambda)\nu_j(\di\lambda)$ with $z_j>0$. Hence $\Pi_j$ is a point mass exactly when $\nu_j$ is. If $Q_j$ has a single root, the support interval is a single point and $\Pi_j$ is a point mass. Otherwise $\nu_j$ is not a point mass by Proposition~\ref{prop:rate-mixture}. Set $b_{j,\lambda}(z)=\psi_j(z/\lambda)$. Changing variables $u=\lambda t$ in \eqref{eq:raw-moments} gives the first identity below. Substitution into \eqref{eq:single-pole} then gives the second:
\begin{align*}
M_0^{(\alpha_j)}(b_{j,\lambda})&=\int_0^\infty u^{\alpha_j}\e^{-u}\psi_j(u/\lambda)^2\dd u=\lambda^{\alpha_j+1}z_j(\lambda),\\
\kappa_{\lambda,b_{j,\lambda}}^{(\alpha_j)}(t)&=\frac{\lambda^{\alpha_j+1}t^{\alpha_j}\e^{-\lambda t}\psi_j(t)^2}{M_0^{(\alpha_j)}(b_{j,\lambda})}=\frac{\Psi_j(t)\e^{-\lambda t}}{z_j(\lambda)}.
\end{align*}
This yields \eqref{eq:decomposition}.
\end{proof}

\begin{example}\label{ex:rate-mixture}
For the two terminal components in Example~\ref{ex:splitting}, Lemma~\ref{lem:contacts} gives $\Psi_j(t)=(t-\xi)^2(t-\zeta_j)^2$, so $\alpha_j=0$, $m_j=2$, and $\psi_j(t)=(t-\xi)(t-\zeta_j)$ ($j=1,2$). Rather than evaluating the determinant in Lemma~\ref{lem:rate-gaps}, we can read off the rate measure $\nu_j$ directly from the factorizations in Example~\ref{ex:splitting}. For $\theta>0$, let $q_\theta(u)=\frac{\theta\e^{\theta u}}{\e^\theta-1}\mathbf1_{(0,1)}(u)$. Direct integration gives
\[
1-\e^{\theta-t}=(t-\theta)\int_0^1\e^{\theta u}\e^{-tu}\dd u=\frac{\e^\theta-1}{\theta}(t-\theta)\int_0^1\e^{-tu}q_\theta(u)\dd u.
\]
Taking independent $U_1,U_2\sim q_\xi$ and $V_1,V_2\sim q_{\zeta_j}$ and setting $A_j=1+U_1+U_2+V_1+V_2$, we obtain
\[
f_j(t)=c_j\Psi_j(t)\E[\e^{-tA_j}],\qquad c_j=\frac{C_j(d_j-1)^2}{\xi^2\zeta_j^2},
\]
where $d_1=3$, $d_2=4$, $C_1=15/8$, and $C_2=15/17$. By the uniqueness in Proposition~\ref{prop:rate-mixture}, $\nu_j$ is the law of $A_j$, which has a positive density on $(1,5)$. Normalizing via $\Pi_j(\di\lambda)=c_jz_j(\lambda)\nu_j(\di\lambda)$ yields \eqref{eq:decomposition} with $b_{j,\lambda}(z)=(z/\lambda-\xi)(z/\lambda-\zeta_j)$.

To illustrate Lemma~\ref{lem:mixture} on $f=\frac12f_1+\frac12f_2$, termwise integration gives means $\mu_1=271/480\approx0.5646$ and $\mu_2=76/255\approx0.2980$, with SCVs $v_1\approx4.5396$, $v_2\approx7.7001$, and $\SCV(f)\approx5.8230$. Taking $v_*=v_1$ in \eqref{eq:mixture-bound} separates the excess SCV into its losses from unequal component SCVs and unequal component means, giving
\[
\SCV(f)-v_1=\frac{\frac12(v_2-v_1)\mu_2^2+(1+v_1)(\mu_1-\mu_2)^2/4}{((\mu_1+\mu_2)/2)^2}\approx0.7546+0.5289=1.2835.
\]
Inside each branch $j=1,2$, the rate law $\Pi_j$ has a positive density on $(1,5)$, so \eqref{eq:mean-derivative} and Lemma~\ref{lem:mixture} yield $\SCV(f_j)>\min_{1\leq\lambda\leq5}\SCV(\kappa_{\lambda,b_{j,\lambda}}^{(0)})$.
\end{example}

\subsection{Optimality across orders}

The optimization problem \eqref{eq:optimization} allows every order up to $N$, so the fixed-order polynomial minima of Section~\ref{sec:single-pole} still have to be compared. A lower-order law cannot remain optimal when one more order is allowed, because averaging it with an independent exponential reduces variance while adding at most one to its ME order. The next lemma gives this construction.

\begin{lemma}\label{lem:attainment}
Let $X$ have an ME law of order $r$ with real spectrum, mean one, and variance $v>0$, and let $E$ be an independent exponential variable with mean one. Then $X'=\frac{1}{1+v}X+\frac{v}{1+v}E$ has a real-spectrum ME law of order at most $r+1$, mean one, and variance $v/(1+v)<v$.
\end{lemma}

\begin{proof}
For $0<c<1$, independence gives $\Var(cX+(1-c)E)=c^2v+(1-c)^2$, which is minimized at $c=1/(1+v)$. At this value the variance is $v/(1+v)$ and the mean remains one. The Laplace transform is $\E[\e^{-csX}]/(1+(1-c)s)$, so the exponential adds at most one degree to the denominator and all poles remain real.
\end{proof}

\begin{proof}[Proof of Theorem~\ref{thm:optimizers}]
We first record that every candidate value is attained. By Theorem~\ref{thm:polynomial} and Lemma~\ref{lem:pair-moments}, for every $k\geq0$ and $\beta\in\{0,1\}$ the value $\Delta_{k+2}^{(\beta)}$ is attained by a mean-one density of the form \eqref{eq:optimizer-density}, with $b^*,\lambda^*$ defined by any minimizing pair of zeros of $L_{k+2}^{(\beta)}$. Its rate is the unit-rate mean $2xy/(x+y)$, and its polynomial has degree $k$, so the exact-order calculation following \eqref{eq:single-pole} gives exact order $2k+1+\beta$.

For a matching lower bound, let $f$ be any real-spectrum ME density of order at most $N$, and write $f=\sum_{j=1}^J\omega_jf_j$ as in Theorem~\ref{thm:decomposition}. For each $j$ and $\lambda>0$, the conditional density $\kappa_{\lambda,b_{j,\lambda}}^{(\alpha_j)}$ has SCV $V_{\alpha_j}(b_{j,\lambda})\geq\Delta_{m_j+2}^{(\alpha_j)}$ by Theorem~\ref{thm:polynomial}. Let $B$ be a random branch with $\mathbb P(B=j)=\omega_j$, and conditionally on $B=j$ draw $\Lambda\sim\Pi_j$. Applying Lemma~\ref{lem:mixture} with mixing variable $(B,\Lambda)$ gives
\[
\SCV(f)\geq\min_{1\leq j\leq J}\Delta_{m_j+2}^{(\alpha_j)}\geq\min_{\substack{k\geq0,\ \beta\in\{0,1\}\\2k+1+\beta\leq N}}\Delta_{k+2}^{(\beta)}.
\]
The explicit candidates attain the latter finite minimum. Hence
\begin{equation}\label{eq:finite-candidates}
\mathfrak v_N^{\mathrm R}=\min_{\substack{k\geq0,\ \beta\in\{0,1\}\\2k+1+\beta\leq N}}\Delta_{k+2}^{(\beta)},
\end{equation}
and the full-class infimum is attained.

At $N=1$ there is only one candidate in \eqref{eq:finite-candidates}. For $N\geq2$, apply Lemma~\ref{lem:attainment} to a mean-one optimizer for order at most $N-1$. Its variance is positive, so this gives \eqref{eq:order-monotonicity}. Since \eqref{eq:finite-candidates} also gives
\[
\mathfrak v_N^{\mathrm R}=\min\bigl(\mathfrak v_{N-1}^{\mathrm R},\Delta_{m+2}^{(\alpha)}\bigr),
\]
strict decrease implies \eqref{eq:minimum} and excludes optimizers of order less than $N$.

Now suppose $f$ attains $\mathfrak v_N^{\mathrm R}$ at mean one. Every conditional density indexed by $(B,\Lambda)$ has SCV at least $\mathfrak v_N^{\mathrm R}$, since rescaling its mean preserves its order and SCV. Since $f$ has mean one, equality in Lemma~\ref{lem:mixture} forces the conditional SCV to equal $\mathfrak v_N^{\mathrm R}$ and the conditional mean to equal one almost surely. Each branch has positive weight, and within branch $j$ the conditional mean is strictly decreasing in $\lambda$ by \eqref{eq:mean-derivative}, so $\Pi_j$ is a point mass and $f_j$ is itself a mean-one optimizer. Strict decrease with order now gives $r_j=N$, hence $m_j=m$ and $\alpha_j=\alpha$.

The reduced denominator of each $f_j$ divides that of $f$ by Theorem~\ref{thm:decomposition}. Since each $f_j$ has order $N$ and $f$ has order at most $N$, these monic denominators all coincide. Because every $\Pi_j$ is a point mass, each $f_j$ has a single pole, so the common denominator is $(s+\lambda)^N$ for a single rate $\lambda>0$. Theorem~\ref{thm:polynomial} and Lemma~\ref{lem:pair-moments} then identify each $f_j$ with a density in \eqref{eq:optimizer-density} for some minimizing pair.

Finally, for any minimizing pair with left index $\ell$, its mean-one rate is
\[
\frac{2x_{\ell,m+2}^{(\alpha)}x_{\ell+1,m+2}^{(\alpha)}}{x_{\ell,m+2}^{(\alpha)}+x_{\ell+1,m+2}^{(\alpha)}}.
\]
Both endpoints increase with $\ell$, and $2xy/(x+y)$ increases in each positive argument, so the rate determines the pair uniquely. Because all components share the pole $-\lambda$, they coincide and $f$ equals that single density.
\end{proof}

The probabilistic convolution in Lemma~\ref{lem:attainment} yields a purely analytic consequence for the Laguerre zeros themselves. Increasing the polynomial degree by one corresponds to two additional ME orders, so applying the exponential construction twice bounds the decrease of the smallest relative gap.

\begin{corollary}\label{cor:monotonicity}
For $m\geq0$ and $\alpha\in\{0,1\}$,
\[
\Delta_{m+3}^{(\alpha)}\leq\frac{\Delta_{m+2}^{(\alpha)}}{1+2\Delta_{m+2}^{(\alpha)}}<\Delta_{m+2}^{(\alpha)}.
\]
\end{corollary}

\begin{proof}
Let $X$ have a mean-one minimizing law of order $N=2m+1+\alpha$, with variance $\Delta_{m+2}^{(\alpha)}$. Applying Lemma~\ref{lem:attainment} twice gives a mean-one real-spectrum ME law of order at most $N+2$ and variance $\Delta_{m+2}^{(\alpha)}/(1+2\Delta_{m+2}^{(\alpha)})$. The minimum at order $N+2$ is $\Delta_{m+3}^{(\alpha)}$ by Theorem~\ref{thm:optimizers}. The second inequality follows from positivity of the variance.
\end{proof}

\section{Two-sided bounds and localization}\label{sec:asymptotics}

By Theorem~\ref{thm:optimizers}, the optimal SCV is the smallest relative gap between adjacent Laguerre zeros. The estimates below establish the bounds and localization in Theorem~\ref{thm:localization}. Because this relative gap depends only on the ratio $y/x$ of two adjacent zeros, we pass to a logarithmic scale, where ratios become differences. For fixed $\alpha\in\{0,1\}$ and $m\geq0$, define
\[
z_{j,m+2}^{(\alpha)}=\log x_{j,m+2}^{(\alpha)},\qquad \ell_{m,\alpha}=\min_{1\leq j<m+2}\bigl(z_{j+1,m+2}^{(\alpha)}-z_{j,m+2}^{(\alpha)}\bigr).
\]
The identity $G(x,y)=\sinh^2(\tfrac12\log(y/x))$ from Section~\ref{sec:single-pole} and the monotonicity of $u\mapsto\sinh^2(u/2)$ on $(0,\infty)$ then give
\begin{equation}\label{eq:log-gap}
\Delta_{m+2}^{(\alpha)}=\sinh^2(\ell_{m,\alpha}/2).
\end{equation}

To estimate $\ell_{m,\alpha}$, we use the differential equation satisfied by $L_{m+2}^{(\alpha)}$ \cite[Eq.~(5.1.2)]{szego1975},
\begin{equation}\label{eq:laguerre-ode}
t(L_{m+2}^{(\alpha)})''(t)+(\alpha+1-t)(L_{m+2}^{(\alpha)})'(t)+(m+2)L_{m+2}^{(\alpha)}(t)=0.
\end{equation}
We next transform \eqref{eq:laguerre-ode} into a simpler form with zeros at $z_{j,m+2}^{(\alpha)}$, which will allow us to bound the distances between consecutive zeros.
\begin{lemma}\label{lem:normal-form}
Set
\begin{equation}\label{eq:transformed-polynomial}
w_{m,\alpha}(s)=\exp\!\left(\frac{\alpha s-\e^s}{2}\right)L_{m+2}^{(\alpha)}(\e^s),\qquad s\in\R,
\end{equation}
and let
\[
\gamma=m+2+\frac{\alpha+1}{2},\qquad q_{m,\alpha}=\sqrt{\gamma^2-\frac{\alpha^2}{4}}.
\]
Define the coefficient $K_{m,\alpha}(s)=\gamma\e^s-\e^{2s}/4-\alpha^2/4$. The zeros of $w_{m,\alpha}$ are the $z_{j,m+2}^{(\alpha)}$, and
\begin{equation}\label{eq:normal-ode}
w_{m,\alpha}''+K_{m,\alpha}w_{m,\alpha}=0.
\end{equation}
Moreover,
\begin{equation}\label{eq:coefficient-bound}
K_{m,\alpha}(s)=q_{m,\alpha}^2-\frac14(\e^s-2\gamma)^2\leq q_{m,\alpha}^2.
\end{equation}
\end{lemma}

\begin{proof}
Put $t=\e^s$ and $u(s)=L_{m+2}^{(\alpha)}(\e^s)$. By the chain rule,
\[
u'(s)=t(L_{m+2}^{(\alpha)})'(t),\qquad u''(s)=t(L_{m+2}^{(\alpha)})'(t)+t^2(L_{m+2}^{(\alpha)})''(t).
\]
Multiplying \eqref{eq:laguerre-ode} by $t$ and substituting $t^2(L_{m+2}^{(\alpha)})''=u''-u'$ therefore gives
\[
u''+(\alpha-t)u'+(m+2)tu=0.
\]
The coefficient of $u'$ in this equation is $\alpha-t$. To remove the first derivative, we seek $u=\e^{g(s)}w$ with $2g'+\alpha-t=0$. The choice $g(s)=(t-\alpha s)/2$ satisfies this condition, since $t=\e^s$. Here $w=w_{m,\alpha}$, $g'(s)=(t-\alpha)/2$, and $g''(s)=t/2$, so
\begin{align*}
u'&=\e^g(w'+g'w),\\
u''&=\e^g\bigl(w''+2g'w'+(g''+(g')^2)w\bigr).
\end{align*}
Since $\alpha-t=-2g'$, substituting into the equation for $u$ cancels the $w'$ term and turns $g''+(g')^2+(\alpha-t)g'$ into $g''-(g')^2$. Dividing by $\e^g$ leaves
\[
w''+\left((m+2)t+\frac t2-\frac{(t-\alpha)^2}{4}\right)w=0.
\]
Expanding $(t-\alpha)^2/4$ collects the coefficient of $t$ into $m+2+(\alpha+1)/2=\gamma$, giving \eqref{eq:normal-ode}. Since the multiplying factor in \eqref{eq:transformed-polynomial} is positive, $w_{m,\alpha}(s)=0$ if and only if $\e^s$ is a zero of $L_{m+2}^{(\alpha)}$, that is, $s=\log x_{j,m+2}^{(\alpha)}=z_{j,m+2}^{(\alpha)}$. Completing the square via
\[
\gamma t-\frac{t^2}{4}=\gamma^2-\frac14(t-2\gamma)^2
\]
and recalling $q_{m,\alpha}^2=\gamma^2-\alpha^2/4$ gives \eqref{eq:coefficient-bound}.
\end{proof}

For comparison, consider the constant-coefficient equation $w''+q^2w=0$ with $q>0$. Its nonzero solutions have the form $A\cos(qs)+B\sin(qs)$, so consecutive zeros are exactly $\pi/q$ apart. In \eqref{eq:coefficient-bound}, the coefficient $K_{m,\alpha}$ is largest at $\e^s=2\gamma$, with maximum $q_{m,\alpha}^2$. We use the following comparison to bound the logarithmic gaps from below and above, and in Section~\ref{sec:computation} to locate the minimizing pairs.

\begin{lemma}\label{lem:zero-spacing}
Let $u,v$ be nonzero real solutions of $u''+Au=0$ and $v''+Bv=0$ on an interval, where $A,B$ are real and continuous. If $a<b$ are consecutive zeros of $v$ and $A>B$ on $(a,b)$, then $u$ has a zero in $(a,b)$. In particular, if $q>0$ and $A>q^2$ on an interval of length greater than $2\pi/q$, that interval contains two consecutive zeros of $u$ less than $\pi/q$ apart.
\end{lemma}

\begin{proof}
If $u$ has no zero in $(a,b)$, we may choose the signs so that both solutions are positive there. Their Wronskian $W=u'v-uv'$ satisfies $W'=(B-A)uv<0$. Yet $W(a)=-u(a)v'(a)\leq0$ and $W(b)=-u(b)v'(b)\geq0$, a contradiction. This argument includes the case where $u$ vanishes at an endpoint.

For the second assertion take $v(s)=\sin(q(s-c))$. Every closed subinterval $[c,c+\pi/q]$ on which $A>q^2$ contains an interior zero of $u$. Choose $[d,d+2\pi/q]$ inside the given interval and let $z$ be such a zero in $(d,d+\pi/q)$. Applying the same argument on $[z,z+\pi/q]$ gives a second zero in $(z,z+\pi/q)$. The zeros are isolated by uniqueness for the differential equation, so the first zero after $z$ gives the desired consecutive pair.
\end{proof}

Applying Lemma~\ref{lem:zero-spacing} with the global upper bound \eqref{eq:coefficient-bound} on $K_{m,\alpha}$ gives a lower bound for every logarithmic gap.

\begin{lemma}\label{lem:gap-lower}
For $m\geq0$ and $\alpha\in\{0,1\}$,
\begin{equation}\label{eq:gap-lower}
\ell_{m,\alpha}\geq\frac{\pi}{q_{m,\alpha}}.
\end{equation}
\end{lemma}

\begin{proof}
By Lemma~\ref{lem:normal-form}, $w_{m,\alpha}$ solves \eqref{eq:normal-ode} and its zeros are the $z_{j,m+2}^{(\alpha)}$. Let $a<b$ be consecutive zeros of $w_{m,\alpha}$, and suppose for contradiction that $\max_{s\in[a,b]}K_{m,\alpha}(s)<\pi^2/(b-a)^2$. Then $u(s)=\sin\bigl(\pi(s-a)/(b-a)\bigr)$ solves $u''+(\pi/(b-a))^2u=0$ with strictly larger coefficient than $K_{m,\alpha}$ on $(a,b)$, so Lemma~\ref{lem:zero-spacing} forces $u$ to vanish in $(a,b)$, a contradiction. Combining this with \eqref{eq:coefficient-bound} gives
\begin{equation}\label{eq:interval-gap}
\frac{\pi^2}{(b-a)^2}\leq\max_{s\in[a,b]}K_{m,\alpha}(s)\leq q_{m,\alpha}^2.
\end{equation}
Taking the minimum of $b-a$ over all consecutive zero intervals proves the result.
\end{proof}

To bound the shortest gap from above, we look for a region around the maximum of $K_{m,\alpha}$ where Lemma~\ref{lem:zero-spacing} forces two nearby zeros. The level needed for the finite-order upper bound can be written directly in terms of $N$.

\begin{lemma}\label{lem:gap-upper}
Let $N=2m+1+\alpha\geq9$ with $\alpha\in\{0,1\}$, and put $r=\tfrac12\sqrt{(N+2)^2-\alpha^2}$. Then
\begin{equation}\label{eq:gap-upper}
\ell_{m,\alpha}<\frac{\pi}{r}\leq\ell_{m-1,\alpha}.
\end{equation}
\end{lemma}

\begin{proof}
Put $T=N+4=2\gamma$ and $q=q_{m,\alpha}=\tfrac12\sqrt{T^2-\alpha^2}$. Since $q^2-r^2=T-1$, the completed square in \eqref{eq:coefficient-bound} gives
\[
K_{m,\alpha}(s)>r^2\quad\Longleftrightarrow\quad |\e^s-T|<2\sqrt{T-1}.
\]
Write $a=\sqrt{T-1}>1$. The endpoints in the original variable $t=\e^s$ are $T\pm2a=(a\pm1)^2$, so the length $L$ of this interval on the logarithmic scale is
\[
L=2\log\frac{a+1}{a-1}=2\int_{a-1}^{a+1}\frac{\dd u}{u}>\frac4a,
\]
by strict convexity of $1/u$. Using $\alpha^2\leq1$ and $N\geq9$, we obtain
\[
\frac{rL}{2}>\frac{2r}{a}=\sqrt{\frac{(T-2)^2-\alpha^2}{T-1}}
\geq\sqrt{T-3}=\sqrt{N+1}\geq\sqrt{10}>\pi.
\]
Lemma~\ref{lem:zero-spacing} now gives a logarithmic gap shorter than $\pi/r$. Finally, $r=q_{m-1,\alpha}$, so Lemma~\ref{lem:gap-lower} at index $m-1$ gives the other inequality in \eqref{eq:gap-upper}.
\end{proof}

\begin{proof}[Proof of Theorem~\ref{thm:localization}]
Write $T=N+4$, $q=q_{m,\alpha}=\tfrac12\sqrt{T^2-\alpha^2}$, $K=K_{m,\alpha}$, and $\ell=\ell_{m,\alpha}$. By \eqref{eq:minimum} and \eqref{eq:log-gap}, $\mathfrak v_N^{\mathrm R}=\sinh^2(\ell/2)$. Lemma~\ref{lem:gap-lower} gives the lower bound in \eqref{eq:variance-bounds} for every $N\geq1$. For $N\geq9$, Lemma~\ref{lem:gap-upper} gives the upper bound.

For the expansion, we need a short gap in a smaller neighborhood of the maximum. Put $s_0=\log T$ and $J=[s_0-2\pi/q,s_0+2\pi/q]$. Since $T/q\to2$, we have $|\e^s-T|\leq T(\e^{2\pi/q}-1)=O(1)$ uniformly on $J$. By \eqref{eq:coefficient-bound}, there is a fixed $C>0$ such that $K>q^2-C$ throughout $J$ for all sufficiently large $N$, for both parities. Put $k=\sqrt{q^2-C}$. As $k/q\to1$, the length $4\pi/q$ of $J$ exceeds $2\pi/k$ for all sufficiently large $N$. Lemma~\ref{lem:zero-spacing} therefore gives a gap shorter than $\pi/k$ in $J$. The shortest gap over all zeros consequently satisfies
\begin{equation}\label{eq:sharp-log-gap}
\frac{\pi}{q}\leq\ell<\frac{\pi}{\sqrt{q^2-C}},\qquad \ell=\frac{\pi}{q}+O(q^{-3}).
\end{equation}
Using $\sinh^2(\ell/2)=\ell^2/4+O(\ell^4)$ gives
\[
\mathfrak v_N^{\mathrm R}=\frac{\pi^2}{4q^2}+O(q^{-4})
=\frac{\pi^2}{(N+4)^2-\alpha^2}+O(N^{-4}).
\]
Expanding the denominator gives \eqref{eq:variance-expansion}.

Now take any minimizing pair $x_m<y_m$ and choose $s_*$ where $K$ is maximal on $[\log x_m,\log y_m]$. This interval has length $\ell$, so \eqref{eq:interval-gap} and \eqref{eq:sharp-log-gap} give
\[
q^2-\tfrac14(\e^{s_*}-T)^2=K(s_*)\geq\frac{\pi^2}{\ell^2}>q^2-C.
\]
Hence $\e^{s_*}=T+O(1)$. Both endpoints in the original variable lie between $\e^{s_*}\e^{-\ell}$ and $\e^{s_*}\e^\ell$. Since $\ell=O(1/q)$ and $T/q\to2$, both are $T+O(1)$. In particular $x_m/q\to2$, and \eqref{eq:sharp-log-gap} gives $q\ell\to\pi$. Thus
\[
y_m-x_m=\frac{x_m}{q}\,(q\ell)\,\frac{\e^\ell-1}{\ell}\longrightarrow 2\pi,
\]
which proves \eqref{eq:zero-localization}.
\end{proof}

The location and spacing estimates in \eqref{eq:zero-localization} have different probabilistic consequences. At unit rate, the mean of the optimizing law is the harmonic mean of the two zeros in the minimizing pair, and their separation determines its limiting variance through Lemma~\ref{lem:pair-moments}.

\begin{corollary}\label{cor:optimizer-moments}
For each $N=2m+1+\alpha$, choose any minimizing pair and its polynomial $b^*$ and rate $\lambda^*$ as in Theorem~\ref{thm:optimizers}, and let $T_N^*$ have density $\kappa_{1,b^*}^{(\alpha)}$. As $N\to\infty$,
\[
\frac{\E[T_N^*]}{m}\longrightarrow2,\qquad \Var(T_N^*)\longrightarrow\pi^2.
\]
The optimal mean-one pole rate satisfies $\lambda^*=\E[T_N^*]\sim2m\sim N$, agreeing to leading order with the rate of the mean-one Erlang law with $N$ phases. In particular, for $U_N^*=T_N^*/\E[T_N^*]$, $\Var(N(U_N^*-1))=N^2\mathfrak v_N^{\mathrm R}\to\pi^2$. Thus the standard deviation of $U_N^*$ is asymptotic to $\pi/N$, while that of the mean-one $N$-phase Erlang law is $1/\sqrt N$.
\end{corollary}

\begin{proof}
By Lemma~\ref{lem:pair-moments}, $\E[T_N^*]=\lambda^*=2x_my_m/(x_m+y_m)$ lies between $x_m$ and $y_m$. Both are $N+4+O(1)$ by \eqref{eq:zero-localization}. Thus $\lambda^*/m\to2$ and $\lambda^*/N\to1$. By scale invariance and Theorem~\ref{thm:localization},
\[
\Var(T_N^*)=\left(\frac{\lambda^*}{N}\right)^2 N^2\mathfrak v_N^{\mathrm R}\longrightarrow\pi^2.
\]
The variance of the centered, scaled mean-one optimizer is $N^2\mathfrak v_N^{\mathrm R}$, with the same limit.
\end{proof}

\section{Computing the optimizers}\label{sec:computation}

For an integer order $N\geq1$, write $N=2m+1+\alpha$, where $\alpha\in\{0,1\}$ and $m\geq0$, and set $n=m+2$. By Theorem~\ref{thm:optimizers}, the minimizers are determined by the zeros $0<x_{1,n}^{(\alpha)}<\cdots<x_{n,n}^{(\alpha)}$ of $L_n^{(\alpha)}$. These zeros are the eigenvalues of the symmetric tridiagonal Jacobi matrix $\boldsymbol{J}_n^{(\alpha)}$ with entries
\[
(\boldsymbol{J}_n^{(\alpha)})_{kk}=2k+\alpha-1,\qquad (\boldsymbol{J}_n^{(\alpha)})_{k,k+1}=(\boldsymbol{J}_n^{(\alpha)})_{k+1,k}=\sqrt{k(k+\alpha)},
\]
for $1\leq k\leq n$ on the diagonal and $1\leq k<n$ off the diagonal. This is the Jacobi matrix used in the Golub--Welsch construction of Gaussian quadrature \cite{golub1969,gautschi2004}. By \eqref{eq:relative-gap}, it suffices to choose an index
\[
 j^*\in\operatorname*{arg\,min}_{1\leq j<n}
 \frac{(x_{j+1,n}^{(\alpha)}-x_{j,n}^{(\alpha)})^2}{4x_{j,n}^{(\alpha)}x_{j+1,n}^{(\alpha)}}.
\]
The minimizing polynomial $b^*$ and the mean-one rate $\lambda^*$ are then given by \eqref{eq:optimizer-shape}.

The Jacobi formulation suggests computing all $n$ eigenvalues and comparing the $n-1$ adjacent gaps. The differential equation from Section~\ref{sec:asymptotics} yields a simpler alternative. Its coefficient $K_{m,\alpha}$, given in \eqref{eq:coefficient-bound}, increases up to $s_0=\log(N+4)$ and decreases thereafter. The next proposition shows that the shortest interval between consecutive logarithmic zeros must lie among two neighboring intervals determined by $s_0$.

\begin{proposition}\label{prop:two-gap}
Let $N=2m+1+\alpha$, $\alpha\in\{0,1\}$, and $n=m+2$. Denote by $i$ the number of zeros of $L_n^{(\alpha)}$ strictly smaller than $N+4$, that is,
\[
i=\#\{r:x_{r,n}^{(\alpha)}<N+4\}.
\]
Every minimizing adjacent pair has left index
\[
j^*\in\{i-1,i\}\cap\{1,\ldots,n-1\}.
\]
Consequently, at most three consecutive zeros are required to determine the minimum, regardless of the order $N$.
\end{proposition}

\begin{proof}
Abbreviate $x_j=x_{j,n}^{(\alpha)}$, $z_j=\log x_j$, $\ell_j=z_{j+1}-z_j$, and write $T=N+4=2\gamma$, $s_0=\log T$. At an interior zero $b=z_j$, reflect the solution $w=w_{m,\alpha}$ about $b$. The functions $w(b-u)$ and $w(b+u)$ have their first zeros for $u>0$ at $\ell_{j-1}$ and $\ell_j$, respectively. Their equations have coefficients $K(b-u)$ and $K(b+u)$, respectively, where $K=K_{m,\alpha}$. With $x=\e^b$, \eqref{eq:coefficient-bound} gives
\[
K(b+u)-K(b-u)=x\sinh u\,(T-x\cosh u).
\]
If $x\geq T$, this difference is strictly negative for every $u>0$, including when $x=T$. Lemma~\ref{lem:zero-spacing} applied on $[0,\ell_j]$ forces a zero of the left reflected solution in $(0,\ell_j)$. Thus $\ell_{j-1}<\ell_j$ even when the preceding gap crosses $s_0$.

If $z_{j-1}<z_j<z_{j+1}<s_0$, then $K(b+u)>K(b-u)$ for $0<u<\ell_j$, since $K$ increases to the left of $s_0$. If $\ell_j\geq\ell_{j-1}$, the same comparison on $[0,\ell_{j-1}]$ would force the right solution to vanish before $\ell_{j-1}$. Hence $\ell_j<\ell_{j-1}$.

By the trace of $\boldsymbol J_n^{(\alpha)}$, the mean of its eigenvalues is $n+\alpha<T=2n+\alpha+1$, so $i\geq1$. The left-side comparison excludes indices $j\leq i-2$ and the reflected comparison at $x_j\geq T$ excludes $j\geq i+1$. This leaves $i-1$ and $i$ whenever they are valid indices, including just the first gap when $i=1$ and just the final gap when $i=n$. Since $G(x_j,x_{j+1})=\sinh^2(\ell_j/2)$ increases strictly with $\ell_j$, this proves the assertion for the relative gaps as well.
\end{proof}

When both candidates in Proposition~\ref{prop:two-gap} exist, their ratios are $x_i/x_{i-1}$ and $x_{i+1}/x_i$, abbreviating $x_j=x_{j,n}^{(\alpha)}$. The pair with left index $i-1$ gives the smaller gap if $x_i^2<x_{i-1}x_{i+1}$, and the pair with left index $i$ does so if $x_i^2>x_{i-1}x_{i+1}$. Equality gives two minimizers. Both strict cases occur in the odd-order computations. At $N=5$ one finds $i=j^*=3$, whereas at $N=15$ one finds $i=8$ and $j^*=7$. Thus it suffices to count the eigenvalues below $N+4$ and isolate at most three consecutive eigenvalues.

The numerical results below use a single sequence of odd orders, with the large orders chosen as $N=10^k+1$ for $k=1,\ldots,8,20,100$. This is the nearest odd order above each power of ten. The computations use the full Jacobi spectrum at small orders and the localization of Proposition~\ref{prop:two-gap} with high-precision refinement at large orders. Table~\ref{tab:finite-values} reports the minimum variance, its $N^2$-scaled value, and the difference from the asymptotic limit $\pi^2$. The method remains valid at every integer order, although only this one sequence is displayed.

\begin{table}[htbp]
\centering
\caption{Minimum variance at selected odd ME orders. For the larger entries, $N=10^k+1$, $k=1,\ldots,8,20,100$. The scaled optimal variances approach $\pi^2$, and the fourth column gives the corresponding deficit. The minimum variance and deficit are given to six significant figures, and the scaled variance to eight decimal places.}
\label{tab:finite-values}
\small
\begin{tabular}{@{}rrrr@{}}
\toprule
$N$ & $\mathfrak v_N^{\mathrm R}$ & $N^2\mathfrak v_N^{\mathrm R}$ & $\pi^2-N^2\mathfrak v_N^{\mathrm R}$ \\
\midrule
$1$ & 1.00000 & $1.00000000$ & 8.86960 \\
$3$ & 0.276583 & $2.48924278$ & 7.38036 \\
$5$ & 0.138453 & $3.46133081$ & 6.40827 \\
$9$ & 0.0600486 & $4.86393861$ & 5.00567 \\
$10^{1}+1$ & 0.0449173 & $5.43498775$ & 4.43462 \\
$10^{2}+1$ & $8.95851\times10^{-4}$ & $9.13857824$ & 0.731026 \\
$10^{3}+1$ & $9.77169\times10^{-6}$ & $9.79124207$ & 0.0783623 \\
$10^{4}+1$ & $9.85974\times10^{-8}$ & $9.86171516$ & 0.00788925 \\
$10^{5}+1$ & $9.86862\times10^{-10}$ & $9.86881490$ & $7.89505\times10^{-4}$ \\
$10^{6}+1$ & $9.86951\times10^{-12}$ & $9.86952544$ & $7.89562\times10^{-5}$ \\
$10^{7}+1$ & $9.86959\times10^{-14}$ & $9.86959651$ & $7.89568\times10^{-6}$ \\
$10^{8}+1$ & $9.86960\times10^{-16}$ & $9.86960361$ & $7.89568\times10^{-7}$ \\
$10^{20}+1$ & $9.86960\times10^{-40}$ & $9.86960440$ & $7.89568\times10^{-19}$ \\
$10^{100}+1$ & $9.86960\times10^{-200}$ & $9.86960440$ & $7.89568\times10^{-99}$ \\
\bottomrule
\end{tabular}
\end{table}

At every computed odd order through $N=1999$ and at each selected higher order, the smallest of the candidate gaps was unique within numerical precision. Already at $N=10^8+1$, $N^2\mathfrak v_N^{\mathrm R}\approx9.8696036115$ is close to $\pi^2\approx9.8696044011$, and at $N=10^{20}+1$ and $N=10^{100}+1$ all eight reported decimal places agree with $\pi^2$. Equation~\eqref{eq:variance-expansion} gives $\pi^2-N^2\mathfrak v_N^{\mathrm R}=8\pi^2/N+O(N^{-2})$ (with $8\pi^2\approx78.9568$), which explains the size of the deficit in the fourth column of Table~\ref{tab:finite-values}.

Figure~\ref{fig:bounds} compares the optimal variance with the Erlang benchmark $1/N$ through order $N=10^8+1$, alongside the theoretical two-sided bounds in \eqref{eq:variance-bounds}. It also shows the scaled minimum and its deficit from $\pi^2$. Logarithmic horizontal axes display the finite-order and large-order regimes together, with logarithmic vertical axes also used for the direct variance comparison and the deficit. The bounds are evaluated with $\alpha=0$, corresponding to the odd orders displayed.

\begin{figure}[!htbp]
\centering
\includegraphics[width=0.96\textwidth]{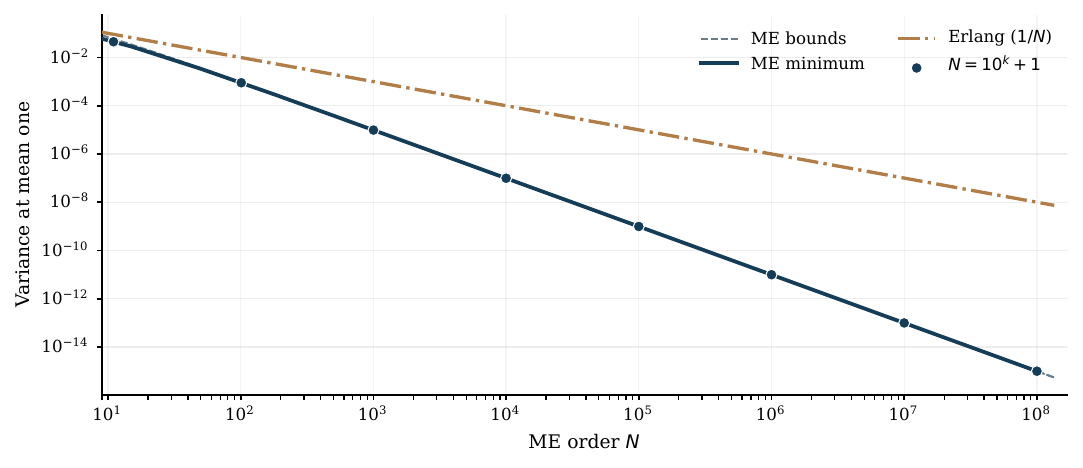}
\medskip
\includegraphics[width=0.48\textwidth]{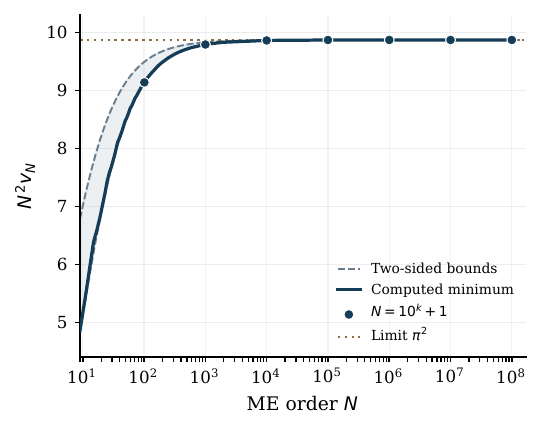}\hfill
\includegraphics[width=0.48\textwidth]{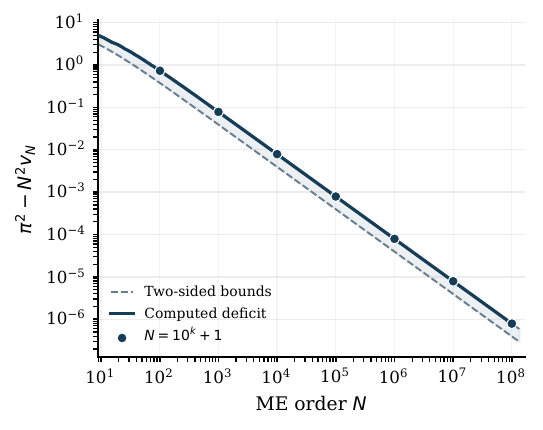}
\caption{Variance comparison and asymptotic behavior at odd ME orders through $N=10^8+1$. The top panel shows the minimum variance $\mathfrak v_N^{\mathrm R}$ and the mean-one $N$-phase Erlang benchmark $1/N$, with the bounds from \eqref{eq:variance-bounds} (dashed and shaded) on logarithmic axes. The lower left panel shows the scaled minimum $N^2\mathfrak v_N^{\mathrm R}$ and its scaled bounds on a logarithmic horizontal axis, with the limit $\pi^2$ shown by a dotted line. The lower right panel shows the deficit $\pi^2-N^2\mathfrak v_N^{\mathrm R}$ on logarithmic axes, with bounds obtained by subtracting the scaled variance bounds from $\pi^2$. Curves connect the computed odd orders through $N=1999$ and the selected higher orders up to $N=10^8+1$. Markers indicate $N=10^k+1$, $k=1,\ldots,8$.}
\label{fig:bounds}
\end{figure}


\newpage

\begin{thebibliography}{99}\raggedright

\bibitem{aldousshepp1987} D.~Aldous and L.~Shepp, The least variable phase type distribution is Erlang, \emph{Comm. Statist. Stochastic Models} \textbf{3} (1987), 467--473.

\bibitem{asmussenavramusabel2002} S.~Asmussen, F.~Avram, and M.~Usabel, Erlangian approximations for finite-horizon ruin probabilities, \emph{ASTIN Bull.} \textbf{32} (2002), 267--281.

\bibitem{avramkyprianoupistorius2004} F.~Avram, A.~E.~Kyprianou, and M.~R.~Pistorius, Exit problems for spectrally negative L\'evy processes and applications to (Canadized) Russian options, \emph{Ann. Appl. Probab.} \textbf{14} (2004), 215--238.

\bibitem{battagliolaperalta2026} M.~L.~Battagliola and O.~Peralta, Optimization-free concentrated matrix-exponentials, \emph{Probab. Engrg. Inform. Sci.}, to appear.

\bibitem{bladtivanovs2021} M.~Bladt and J.~Ivanovs, Fluctuation theory for one-sided L\'evy processes with a matrix-exponential time horizon, \emph{Stochastic Process. Appl.} \textbf{142} (2021), 105--123.

\bibitem{bladtnielsen2017} M.~Bladt and B.~F.~Nielsen, \emph{Matrix-Exponential Distributions in Applied Probability}, Probability Theory and Stochastic Modelling, vol.~81, Springer, New York, 2017.

\bibitem{bladtnielsenperalta2019} M.~Bladt, B.~F.~Nielsen, and O.~Peralta, Parisian types of ruin probabilities for a class of dependent risk-reserve processes, \emph{Scand. Actuar. J.} \textbf{2019} (2019), 32--61.

\bibitem{bouchardelkarouitouzi2005} B.~Bouchard, N.~El~Karoui, and N.~Touzi, Maturity randomization for stochastic control problems, \emph{Ann. Appl. Probab.} \textbf{15} (2005), 2575--2605.

\bibitem{carr1998} P.~Carr, Randomization and the American put, \emph{Rev. Financ. Stud.} \textbf{11} (1998), 597--626.

\bibitem{gautschi2004} W.~Gautschi, \emph{Orthogonal Polynomials: Computation and Approximation}, Oxford University Press, Oxford, 2004.

\bibitem{golub1969} G.~H.~Golub and J.~H.~Welsch, Calculation of Gauss quadrature rules, \emph{Math. Comp.} \textbf{23} (1969), 221--230.

\bibitem{karlinmcgregor1959} S.~Karlin and J.~McGregor, Coincidence probabilities, \emph{Pacific J. Math.} \textbf{9} (1959), 1141--1164.

\bibitem{karlinstudden1966} S.~Karlin and W.~J.~Studden, \emph{Tchebycheff Systems: With Applications in Analysis and Statistics}, Pure and Applied Mathematics, vol.~15, Interscience Publishers, New York, 1966.

\bibitem{kuznetsovetal2011} A.~Kuznetsov, A.~E.~Kyprianou, J.~C.~Pardo, and K.~van Schaik, A Wiener--Hopf Monte Carlo simulation technique for L\'evy processes, \emph{Ann. Appl. Probab.} \textbf{21} (2011), 2171--2190.

\bibitem{kyprianoupistorius2003} A.~E.~Kyprianou and M.~R.~Pistorius, Perpetual options and Canadization through fluctuation theory, \emph{Ann. Appl. Probab.} \textbf{13} (2003), 1077--1098.

\bibitem{meszarostelek2022} A.~M\'esz\'aros and M.~Telek, Concentrated matrix exponential distributions with real eigenvalues, \emph{Probab. Engrg. Inform. Sci.} \textbf{36} (2022), 1171--1187.

\bibitem{ocinneide1990} C.~A.~O'Cinneide, Characterization of phase-type distributions, \emph{Comm. Statist. Stochastic Models} \textbf{6} (1990), 1--57.

\bibitem{ocinneide1991} C.~A.~O'Cinneide, Phase-type distributions and majorization, \emph{Ann. Appl. Probab.} \textbf{1} (1991), 219--227.

\bibitem{peralta2023} O.~Peralta, A Markov jump process associated with the matrix-exponential distribution, \emph{J. Appl. Probab.} \textbf{60} (2023), 1--13.

\bibitem{szego1975} G.~Szeg\H{o}, \emph{Orthogonal Polynomials}, 4th ed., American Mathematical Society Colloquium Publications, vol.~23, American Mathematical Society, Providence, RI, 1975.

\end{thebibliography}
\end{document}